\documentclass[11pt,a4paper]{article}
\usepackage[utf8]{inputenc}
\usepackage[english]{babel}
\usepackage{pst-grad} 
\usepackage{pst-plot} 
\usepackage{pstricks}
\usepackage{amsmath,amssymb,mathrsfs,amsthm}
\usepackage[dvips]{graphicx}
\usepackage{anysize} 
\marginsize{3cm}{2cm}{2cm}{2cm} 
\usepackage{epsfig}

\newcommand{\RR}{\mathbb R}

\newcommand{\pat}{\partial_t}
\newcommand{\pax}{\partial_x}

\newcommand{\jeps}{\mathcal{J}_\epsilon*}

\newcommand{\feps}{f^\epsilon}
\usepackage[pdfauthor={Rafael Granero Belinch\'on},pdftitle={},bookmarks,colorlinks]{hyperref}
\hypersetup{colorlinks=true}
\newcounter{comentcount}
\newcounter{teocount}
\newcounter{propcount}
\newtheorem{lem}{Lemma}
\newtheorem{prop}[propcount]{Proposition}

\newtheorem{teo}[teocount]{Theorem}  
\newtheorem{defi}{Definition}

\newenvironment{coment}
{\stepcounter{comentcount} {\bf \tt Remark} {\bf\tt\arabic{comentcount}} }{ }
\title{Global existence for the confined Muskat problem}
\author{Rafael Granero-Belinch\'on$^{\mbox{{\footnotesize 1}}}$}

\begin{document}

\maketitle 

\footnotetext[1]{Email: \texttt{rgranero@math.ucdavis.edu} Department of Mathematics, University of California at Davis, One Shields Avenue, Davis, CA, 95616.}

\vspace{0.3cm}
\begin{abstract}
In this paper we show global existence of Lipschitz continuous solution for the stable Muskat problem with finite depth (confined) and initial data satisfying some smallness conditions relating the amplitude, the slope and the depth. The cornerstone of the argument is that, for these \emph{small} initial data, both the amplitude and the slope remain uniformly bounded for all positive times. We notice that, for some of these solutions, the slope can grow but it remains bounded. This is very different from the infinite deep case, where the slope of the solutions satisfy a maximum principle. Our work generalizes a previous result where the depth is infinite.
\end{abstract}
\vspace{0.3cm}


\textbf{Acknowledgments}: The author is supported by the grant MTM2011-26696 from Ministerio de Econom\'ia y Competitividad (MINECO). The author thanks David Paredes and Professors Diego C\'ordoba and Rafael Orive for comments that greatly improved the manuscript. The author is grateful to reviewers for their helpful suggestions.

\section{Introduction}
In this paper we study the dynamics of two different incompressible fluids with the same viscosity  in a bounded porous medium. This is known as the confined Muskat problem. For this problem we show that there are global in time Lipschitz continuous solutions corresponding to initial data that fulfills some conditions related to the amplitude, slope and depth. This problem is of practical importance because it is used as a model for a geothermal reservoir (see \cite{CF} and references therein) or a model of an aquifer or an oil well (see  \cite{Muskat}). The velocity of a fluid flowing in a porous medium satisfies Darcy's law (see \cite{bear,Muskat,bn}) 
\begin{equation}
\frac{\mu}{\kappa}v(\vec{x})=-\nabla p(\vec{x})-g\rho(\vec{x}) (0,1),
\label{IVeq1} 
\end{equation}
where $\mu$ is the dynamic viscosity, $\kappa$ is the permeability of the medium, $g$ is the acceleration due to gravity, $\rho(\vec{x})$ is the density of the fluid, $p(\vec{x})$ is the pressure of the fluid and $v(\vec{x})$ is the incompressible velocity field. To simplify the notation we assume $g=\mu/\kappa=1.$ The motion of a fluid in a two-dimensional porous medium is analogous to the Hele-Shaw cell problem (see \cite{cheng2012global, Peter, ES, H-S} and the references therein).

Let us consider the spatial domain $S=\RR\times(-l,l)$ for $0<l$. We assume impermeable boundary conditions for the velocity in the walls. In this domain we have two immiscible and incompressible fluids with the same viscosity and different densities; $\rho^1$ fills the upper subdomain and $\rho^2$ fills the lower subdomain (see Figure \ref{IVscheme}). The graph $f(x,t)$ is the interface between the fluids. 

It is well-known that the system is in the (Rayleigh-Taylor) stable regime if the denser fluid is below the lighter one in every point $\vec{x}$, \emph{i.e.} $\rho^2>\rho^1$. Conversely, the system is in the unstable regime if there is at least a point $\vec{x}$ where the denser fluid is above the lighter one.

\begin{figure}[t]
		\begin{center}
		\includegraphics[scale=0.3]{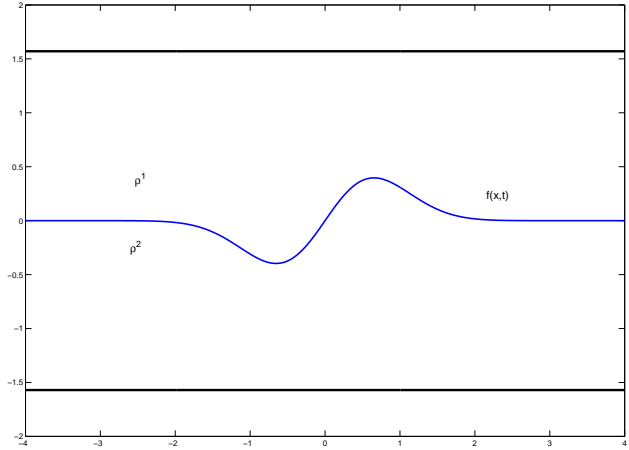} 
		\end{center}
		\caption{Physical situation}
\label{IVscheme}
\end{figure}

If the fluids fill the whole plane the contour equation satisfies (see \cite{c-g07})
\begin{equation}\label{IVfull}
\pat f=\frac{\rho^2-\rho^1}{2\pi}\text{P.V.}\int_\RR \frac{(\pax f(x)-\pax f(x-\eta))\eta}{\eta^2+(f(x)-f(x-\eta))^2}d\eta.
\end{equation}
For this equation the authors show the existence of classical solution locally in time (see \cite{c-g07} and also \cite{ambrose2004well, escher2011generalized, e-m10, KK}) in the Rayleigh-Taylor stable regime, and maximum principles for $\|f(t)\|_{L^\infty(\RR)}$ and $\|\pax f(t)\|_{L^\infty(\RR)}$ (see \cite{c-g09}). Moreover, in \cite{castro2012breakdown, ccfgl} the authors show the existence of turning waves and finite time singularities. In \cite{ccgs-10} the authors show an energy balance for the $L^2$ norm and some results concerning the global existence of solutions corresponding to \emph{'small'} initial data. Furthermore, they show that if initially $\|\pax f_0\|_{L^\infty(\RR)}<1$, then there is global Lipschitz solution and if the initial data has small $H^3$ norm then there is global classical solution.

The case where the fluid domain is the strip $S=\RR\times(-l,l)$, with $0<l$, has been studied in \cite{BCG, CGO, escher2011generalized, e-m10, GG}. In this domain the equation for the interface is
\begin{multline}\label{IVeqderiv}
\pat f(x,t)=\frac{\rho^2-\rho^1}{2\pi}\partial_x\text{P.V.}\int_\RR\arctan\left(\frac{\tan\left(\frac{\pi}{2l}\frac{f(x)-f(x-\eta)}{2}\right)}{\tanh\left(\frac{\pi}{2l}\frac{\eta}{2}\right)}\right)d\eta\\
+\frac{\rho^2-\rho^1}{2\pi}\partial_x\text{P.V.}\int_\RR\arctan\left(\tan\left(\frac{\pi}{2l}\frac{f(x)+f(x-\eta)}{2}\right)\tanh\left(\frac{\pi}{2l}\frac{\eta}{2}\right)\right)d\eta.
\end{multline}

For equation \eqref{IVeqderiv} the authors in \cite{CGO} obtain local existence of classical solution when the system starts its evolution in the stable regime and the initial interface does not reach the walls, and the existence of initial data such that $\|\pax f\|_{L^\infty(\RR)}$ blows up in finite time. The authors also study the effect of the boundaries on the evolution of the interface, obtaining the maximum principle and a decay estimate for $\|f(t)\|_{L^\infty(\RR)}$ and the maximum principle for $\|\pax f(t)\|_{L^\infty(\RR)}$ for initial data satisfying the following hypotheses:
\begin{equation}\label{IVH3}
\|\pax f_0\|_{L^\infty(\RR)}<1,
\end{equation}
\begin{equation}\label{IVH4}
\tan\left(\frac{\pi\|f_0\|_{L^\infty(\RR)}}{2l}\right)< \|\pax f_0\|_{L^\infty(\RR)}\tanh\left(\frac{\pi}{4l}\right),
\end{equation}
and
\begin{multline}\label{IVH5}
\left(\|\pax f_0\|_{L^\infty(\RR)}+|2(\cos\left(\frac{\pi}{2l}\right)-2)\sec^4\left(\frac{\pi}{4l}\right)|\|\pax f_0\|_{L^\infty(\RR)}^3\right)\frac{\pi^3}{8l^3}\\
\times\frac{\left(1+\|\pax f_0\|_{L^\infty(\RR)}\left(\|\pax f_0\|_{L^\infty(\RR)}+\frac{\tan\left(\frac{\pi }{2l}\frac{\|\pax f_0\|_{L^\infty(\RR)}}{2}\right)}{\tanh\left(\frac{\pi }{4l}\right)}\right)\right)}{6\tanh\left(\frac{\pi }{4l}\right)}\frac{\pi^2}{4l^2}\\
+4\tan\left(\frac{\pi }{2l}\|f_0\|_{L^\infty}\right)-4\|\pax f_0\|_{L^\infty(\RR)}\cos\left(\frac{\pi }{l}\|f_0\|_{L^\infty(\RR)}\right)<0
\end{multline}
These hypotheses are smallness conditions relating $\|\pax f_0\|_{L^\infty(\RR)}$, $\|f_0\|_{L^\infty(\RR)}$ and the depth. We define $(x(l),y(l))$ as the solution of the system
\begin{equation}\label{IVsisder}
\left\{ \begin{array}{ll}
         \tan\left(\frac{\pi x}{2l}\right)- y\tanh\left(\frac{\pi}{4l}\right)=0\\
        \left(y+|2(\cos\left(\frac{\pi}{2l}\right)-2)\sec^4\left(\frac{\pi}{4l}\right)|y^3\right)\frac{\left(1+y\left(y+\frac{\tan\left(\frac{\pi }{2l}\frac{y}{2}\right)}{\tanh\left(\frac{\pi }{4l}\right)}\right)\right)}{6\tanh\left(\frac{\pi }{4l}\right)}\left(\frac{\pi}{2l}\right)^5\\
        \qquad\qquad\qquad\qquad\qquad\qquad\qquad\qquad+4\tan\left(\frac{\pi }{2l}x\right)-4y\cos\left(\frac{\pi }{l}x\right)=0.
        \end{array}\right.
\end{equation}
Then, for initial data satisfying 
\begin{equation}\label{IVsisder2}
\|\pax f_0\|_{L^\infty(\RR)}<y(l) \text{ and } \|f_0\|_{L^\infty(\RR)}<x(l),
\end{equation}
the authors in \cite{CGO} show that
$$
\|\pax f\|_{L^\infty(\RR)}\leq 1.
$$
These inequalities define a region where the slope of the solution can grow but it is bounded uniformly in time. This region only appears in the finite depth case. 

In this paper the question of global existence of weak solution (in the sense of Definition \ref{IVdefi}) for \eqref{IVeqderiv} in the stable regime is adressed. In particular we show the following theorem:
\begin{teo}\label{IVglobal}
Let $f_0(x)\in W^{1,\infty}(\RR)$ be the initial datum satisfying hypotheses \eqref{IVH3}, \eqref{IVH4} and \eqref{IVH5} or \eqref{IVsisder2} in the Rayleigh-Taylor stable regime. Then there exists a global solution 
$$
f(x,t)\in C([0,\infty)\times \RR)\cap L^\infty([0,\infty),W^{1,\infty}(\RR)).
$$
Moreover, if the initial data satisfy \eqref{IVH3}, \eqref{IVH4} and \eqref{IVH5} the solution fulfills the following bounds:
$$
\|f(t)\|_{L^\infty(\RR)}\leq\|f_0\|_{L^\infty(\RR)}\text{ and }\|\pax f(t)\|_{L^\infty(\RR)}\leq\|\pax f_0\|_{L^\infty(\RR)},
$$
while, if the initial datums satisfy \eqref{IVsisder2}, the solution satisfies the following bounds:
$$
\|f(t)\|_{L^\infty(\RR)}\leq\|f_0\|_{L^\infty(\RR)}\text{ and }\|\pax f(t)\|_{L^\infty(\RR)}\leq1.
$$
\end{teo}

This result excludes the formation of cusps (blow up of the first and second derivatives) and turning waves for these initial data, remaining open the existence (or non-existence) of corners (blow up of the curvature with finite first derivative) during the evolution. Notice that in the limit $l\rightarrow\infty$ we recover the result contained in \cite{ccgs-10}. In this paper and the works \cite{BCG,CGO,GG} the effect of the boundaries over the evolution of the internal wave in a flow in porous media has been addressed. When these results for the confined case are compared with the known results in the case where the depth is infinite (see \cite{ccfgl, ccgs-10, c-g07, c-g09}) three main differences appear:
\begin{enumerate}
\item the decay of the maximum amplitude is slower in the confined case. 
\item there are smooth curves with finite energy that turn over in the confined case but do not show this behaviour when the fluids fill the whole plane.
\item to avoid the turning effect in the confined case you need to have smallness conditions in $\|f_0\|_{L^\infty(\RR)}$ and $\|\pax f_0\|_{L^\infty(\RR)}$. However, in the unconfined case, only the condition in the slope is required. Moreover, in the confined case a new region without turning effect appears: a region without a maximum principle for the slope but with an uniform bound. In both cases (the region with the maximum principle and the region with the uniform bound), Theorem \ref{IVglobal} ensures the existence of a global Lipschitz continuous solution.
\end{enumerate}
Keeping these results in mind, there are some questions that remain open. For instance, the existence of a wave whose maximum slope grows but remains uniformly bounded, or the existence of a wave with small slope such that, due to the distance to the boundaries, its slope grows and the existence (or non-existence) of corner-like singularities when the initial data considered is small in $W^{1,\infty}(\RR)$.

The proof of Theorem \ref{IVglobal} is achieved using some lemmas and propositions. First, we define \emph{'ad hoc'} diffusive operators and the regularized system (see Section \ref{IVsec1}). For this regularized system, we show some \emph{a priori} bounds for the amplitude and the slope. With these \emph{'a priori'} bounds we show global existence of $H^3$ solution (see Section \ref{IVsecglobal}). Then, we obtain the weak solution to \eqref{IVeqderiv}, $f$, as the limit of the regularized solutions (see Sections \ref{IVsec4} and \ref{IVsec5}).

\begin{coment}
On the rest of the paper we take $\pi/2l=1$ and $\rho^2-\rho^1=4\pi$ and we drop in the notation the $t$ dependence. We write $c$ for a universal constant that can change from one line to another. We denote $B(y,r)=[y-r,y+r].$
\end{coment}

\section{The regularized system}\label{IVsec1}
In this Section we define the regularized system and obtain some useful \emph{'a priori'} bounds for the amplitude and the slope. To clarify the exposition we write $\feps(x,t)$ for the solution of the regularized system. 
\subsection{Motivation and methodology}
We remark that the term
\begin{equation*}
\Xi_1(x,\eta)=\partial_x\arctan\left(\frac{\tan\left(\frac{f(x)-f(x-\eta)}{2}\right)}{\tanh\left(\frac{\eta}{2}\right)}\right)d\eta,
\end{equation*}
in \eqref{IVeqderiv} is a singular integral operator, while
\begin{equation*}
\Xi_2(x,\eta)=\partial_x\arctan\left(\tan\left(\frac{f(x)+f(x-\eta)}{2}\right)\tanh\left(\frac{\eta}{2}\right)\right)d\eta
\end{equation*}
is not if the curve does not reach the boundaries. In order to remove the singularity while preserving the inner structure, we put a term $|\tanh\left(\frac{\eta}{2}\right)|^\epsilon$ for $0<\epsilon<1/10$ in both kernels. We define
\begin{equation}\label{IVxi1eps}
\Xi_1^ \epsilon(x,\eta)=\partial_x\arctan\left(\frac{\tan\left(\frac{\feps(x)-\feps(x-\eta)}{2}\right)|\tanh(\frac{\eta}{2})|^\epsilon}{\tanh\left(\frac{\eta}{2}\right)}\right)d\eta,
\end{equation}
and
\begin{equation}\label{IVxi2eps}
\Xi_2^\epsilon(x,\eta)=\partial_x\arctan\left(\frac{\tan\left(\frac{\feps(x)+\feps(x-\eta)}{2}\right)}{|\tanh(\frac{\eta}{2})|^\epsilon}\tanh\left(\frac{\eta}{2}\right)\right)d\eta,
\end{equation}
To pass to the limit we use compactness coming from an uniform bound in\newline $L^\infty([0,T],W^{1,\infty}(\RR))$. Thus, we need to obtain \emph{'a priori'} bounds for the amplitude and the slope. We define $\alpha_i,$ $i=1,2,3,4$ positive constants that will be fixed below depending only on the initial datum considered. Taking derivatives in $\Xi_i^\epsilon$, we obtain some terms with positive contribution. So, we attach some diffusive operators to the regularized system. Given a smooth function $\phi$, we define
\begin{equation}\label{IVlambdaeps}
\Lambda_l^{1-\epsilon}\phi(x)=\text{PV}\int_\RR\frac{(\phi(x)-\phi(x-\eta))|\tanh\left(\frac{\eta}{2}\right)|^\epsilon}{\sinh^2\left(\frac{\eta}{2}\right)}d\eta.
\end{equation}
We notice that, if the depth is not $l=\pi/2$, the previous operators should be rescaled and we write the subscript $l$ to keep this dependence in mind. These operators are finite depth versions of the classical $\Lambda^\alpha=(-\Delta)^{\alpha/2}$. Roughly speaking, there are three different types of \emph{extra} terms appearing in the derivatives of \eqref{IVxi1eps} and \eqref{IVxi2eps} that we need to control to obtain the \emph{'a priori'} bound for the slope: 
\begin{enumerate}
\item There are terms which have an integrable singularity and they appear multiplied by $\epsilon$. In order to handle these terms we add $-\epsilon \alpha_2\Lambda_l^{1-\epsilon} \feps(x)$ and $-\epsilon\alpha_3\Lambda_l^{1-3\epsilon}\feps(x)$. These two scales $1-\epsilon$, $1-3\epsilon$, appear naturally due to the nonlinearity present in \eqref{IVeqderiv}.
\item There are terms which are nonlinear versions of $\Lambda_l-\Lambda_l^{1-\epsilon}$ and $\Lambda_l-\Lambda_l^{1-3\epsilon}$. These terms go to zero due to the convergence of the operators but they are not multiplied by $\epsilon$. In order to handle these terms we add $-(\Lambda_l-\Lambda_l^{1-\epsilon})\feps(x)$ and $-\alpha_4(\Lambda_l-\Lambda_l^{1-3\epsilon})\feps(x)$.
\item To absorb the nonsingular terms we add $-\sqrt{\epsilon}\alpha_1 \feps(x)$. We notice that, as $\epsilon<1/10$, the square root converges to zero less than linearly. This factor will be used because the contribution of some terms is $O(\epsilon^{a})$ with $1/2<a<1$. 
\end{enumerate}
Once the \emph{'a priori'} bounds are achieved, we should prove global solvability in $H^3$ for the regularized system. To get this bound we add $\epsilon\pax^2 \feps(x)$. We also regularize the initial datum. We take $\mathcal{J}\in C_c^\infty(\RR)$, $\mathcal{J}\geq0$ and $\|\mathcal{J}\|_{L^1}=1$, a symmetric mollifier and define $\mathcal{J}_\epsilon (x)=\mathcal{J}(x/\epsilon)/\epsilon$. Given $f_0\in W^{1,\infty}(\RR)$ we define the initial datum for the regularized system as
\begin{equation}\label{IVinitial}
\feps(x,0)=\frac{\jeps f_0}{1+\epsilon^2x^2}.
\end{equation}
Putting all together, we define the regularized system
\begin{multline}\label{IVeqreg}
\pat f^\epsilon(x)=-\sqrt{\epsilon} \alpha_1 \feps(x)+\epsilon\pax^2\feps(x)-\epsilon \alpha_2\Lambda_l^{1-\epsilon}\feps(x)\\
-\epsilon\alpha_3\Lambda_l^{1-3\epsilon}\feps(x)-(\Lambda_l-\Lambda_l^{1-\epsilon})\feps(x)-
\alpha_4(\Lambda_l-\Lambda_l^{1-3\epsilon})\feps(x)\\
+ 2\text{P.V.}\int_\RR\Xi_1^\epsilon(x,\eta)d\eta+2\text{P.V.}\int_\RR\Xi_2^\epsilon(x,\eta)d\eta.
\end{multline}
where $\alpha_i$ are universal constants that will be fixed below depending only on the initial datum $f_0$. We remark that $\feps_0\in H^k(\RR)$ for all $k\geq 0$. Notice that, due to the continuity of $f_0$, 
$$
\feps_0=\jeps f_0-\epsilon^2x^2\feps_0\rightarrow f_0
$$ 
uniformly on any compact set in $\RR$. Since $\pax f_0\in L^\infty(\RR)$, we get $\pax f_0\in L^1_{loc}(\RR)$ and then, as as $\epsilon\rightarrow0$, we have $\jeps \pax f_0\rightarrow \pax f_0$ a.e.  Thus, we have $\|\jeps f_0\|_{L^{\infty}(\RR)}\rightarrow \|f_0\|_{L^{\infty}(\RR)}$ and $\|\jeps \pax f_0\|_{L^{\infty}(\RR)}\rightarrow \|\pax f_0\|_{L^{\infty}(\RR)}$. Furthermore, we have that if $f_0$ satisfies the hypotheses \eqref{IVH3}, \eqref{IVH4} and \eqref{IVH5}, $\feps_0$ also satisfy these hypotheses if $\epsilon$ is small enough. Moreover, if $f_0,\pax f_0$ satisfy \eqref{IVsisder2} the same remains valid for $\feps$ and $\pax \feps$ if $\epsilon$ is small enough. 

We use some properties of the operators $\Lambda_l^{1-\epsilon}$. For the reader's convenience, we collect them in the following lemma:
\begin{lem}\label{IVops}
For the operators $\Lambda_l^{1-\epsilon}$ (see \eqref{IVlambdaeps}), the following properties hold:
\begin{enumerate}
\item $\Lambda_l^{1-\epsilon}$ is $L^2$-symmetric.
\item $\Lambda_l^{1-\epsilon}$ is positive definite.
\item Let $\phi$ be a Schwartz function. Then, they converge acting on $\phi$ as $\epsilon$ goes to zero:
$$
\|(\Lambda_l-\Lambda_l^{1-\epsilon})\phi\|_{L^1(\RR)}\leq c\|\phi\|_{W^{2,1}(\RR)}\epsilon.
$$
\item Let $\phi$ be a Schwartz function. Then, the derivative can be written in two different forms as
\begin{multline*}
\Lambda_l^{1-\epsilon}\pax \phi(x)=(1-\epsilon)\text{P.V.}\int_\RR\frac{\left(\pax \phi(x)-\frac{\phi(x)-\phi(\eta)}{\sinh(x-\eta)}\right)|\tanh((x-\eta)/2)|^\epsilon}{\sinh^2\left(\frac{x-\eta}{2}\right)}d\eta\\
+\text{P.V.}\int_\RR\frac{\left(\pax \phi(x)-\frac{\phi(x)-\phi(\eta)}{\tanh(x-\eta)}\right)|\tanh((x-\eta)/2)|^\epsilon}{\sinh^2\left(\frac{x-\eta}{2}\right)}d\eta+4\pax \phi(x)\\
=\text{P.V.}\int_\RR\frac{\pax(\phi(x)-\phi(x-\eta))|\tanh(\eta/2)|^\epsilon}{\sinh^2(\frac{\eta}{2})}d\eta
\end{multline*}
\end{enumerate}
\end{lem}
\begin{proof}
The proof of the first two statement follows from \eqref{IVlambdaeps}. For the proof of the third part we recall some useful facts: if $|y|\geq\delta>0$, due to the Mean Value Theorem, we get
\begin{equation}\label{IVfact1}\left||\tanh(y)|^\epsilon-1\right|=\left|\frac{d}{d\gamma}|\tanh(y)|^\gamma\big{|}_{\gamma=\xi}\epsilon\right|\leq \epsilon|\log\left(|\tanh(y)|\right)|,
\end{equation}
and
\begin{equation}\label{IVfact1.b}
\int_0^\infty|\log\left(|\tanh(y)|\right)|dy\leq c<\infty.
\end{equation}
Now the proof follows in a straightforward way. For the last statement we use the cancellation coming from the principal value to define
$$
F_R(x)=\int_{1/R<|x-\eta|<R}\frac{(\phi(x)-\phi(\eta))|\tanh((x-\eta)/2)|^\epsilon}{\sinh^2\left(\frac{x-\eta}{2}\right)}-2\pax\phi(x)\frac{|\tanh((x-\eta)/2)|^\epsilon}{\tanh((x-\eta)/2)}d\eta.
$$
Using the uniform convergence of the derivative, we conclude the result.
\end{proof}

\subsection{Maximum principle for $\feps$}
In this section we prove an \emph{a priori} bound for $\feps$. To simplify notation we define
\begin{equation}\label{IVtheta}
\theta=\frac{\feps(x)-\feps(\eta)}{2} \text{ and } \bar{\theta}=\frac{\feps(x)+\feps(\eta)}{2}.
\end{equation}

\begin{prop}\label{IVMPf}
Let $f_0\in W^{1,\infty}(\RR)$ be the initial datum in \eqref{IVdefi}, define $\feps_0$ as in \eqref{IVinitial} and let $\feps$ be the classical solution of \eqref{IVeqreg} corresponding to the initial datum $\feps_0$. Then $\feps$ verifies
$$
\|\feps(t)\|_{L^\infty(\RR)}\leq \|\feps_0\|_{L^\infty(\RR)} \leq \|f_0\|_{L^\infty(\RR)}.
$$
Moreover, if $f_0$ has a sign then this sign is preserved during the evolution of $\feps$.
\end{prop}
\begin{proof}
Changing variables and taking the derivative we obtain that \eqref{IVeqreg} is equivalent to
\begin{eqnarray}\label{IVeqreg2}
\pat f^\epsilon(x)&=&-(4+\sqrt{\epsilon} \alpha_1) \feps(x)+\epsilon\pax^2\feps(x)-\epsilon \alpha_2\Lambda_{l}^{1-\epsilon}\feps(x)\nonumber\\
&&-\epsilon \alpha_3\Lambda_{l}^{1-3\epsilon}\feps(x)-(\Lambda_l-\Lambda_l^{1-\epsilon})\feps(x)
-\alpha_4(\Lambda_l-\Lambda_l^{1-3\epsilon})\feps(x)\nonumber\\\nonumber
&&+\text{P.V.}\int_\RR\frac{\pax\feps(x)\sec^2(\theta)\frac{|\tanh((x-\eta)/2)|^\epsilon}{\tanh((x-\eta)/2)}+(\epsilon-1)\tan(\theta)\frac{|\tanh((x-\eta)/2)|^\epsilon}{\sinh^2((x-\eta)/2)}}{1+\frac{\tan^2(\theta)|\tanh\left((x-\eta)/2\right)|^{2\epsilon}}{\tanh^2\left((x-\eta)/2\right)}}d\eta\nonumber\\
&&+\text{P.V.}\int_\RR\frac{\pax\feps(x)\sec^2(\bar{\theta})\frac{\tanh((x-\eta)/2)}{|\tanh((x-\eta)/2)|^\epsilon}}{1+\frac{\tan^2(\bar{\theta})\tanh^2\left((x-\eta)/2\right)}{|\tanh\left((x-\eta)/2\right)|^{2\epsilon}}}d\eta\nonumber\\
&&+\text{P.V.}\int_\RR\frac{\text{sech}^2((x-\eta)/2)(1-\epsilon)\tan(\bar{\theta})d\eta}{|\tanh((x-\eta)/2)|^\epsilon\left(1+\frac{\tan^2(\bar{\theta})\tanh^2\left((x-\eta)/2\right)}{|\tanh\left((x-\eta)/2\right)|^{2\epsilon}}\right)},
\end{eqnarray}

If $\|\feps(t)\|_{L^\infty(\RR)}=\max \feps(x,t)$ we define $\feps(x_t)=\|\feps(t)\|_{L^\infty(\RR)}$. Then we have $\pat \feps(x_t)=\frac{d}{dt}\|\feps(t)\|_{L^\infty(\RR)}$ (see \cite{CGO} for the details). If $\|\feps(t)\|_{L^\infty(\RR)}=\min \feps(x,t)$ we write $\feps(x_t)=-\|\feps(t)\|_{L^\infty(\RR)}$ and we get $-\pat \feps(x_t)=\frac{d}{dt}\|\feps(t)\|_{L^\infty(\RR)}$. We compute
\begin{multline*}
4\feps(x)=2\int_\RR\partial_\eta\arctan\left(\tan(\feps(x))\frac{\tanh(\eta/2)}{|\tanh(\eta/2)|^\epsilon}\right)d\eta\\
=-\int_\RR\frac{1}{\cosh^2(\eta/2)}\frac{(\epsilon-1)\tan(\feps(x))|\tanh(\eta/2)|^\epsilon}{|\tanh(\eta/2)|^{2\epsilon}+\tanh^2(\eta/2)\tan^2(\feps(x))}d\eta\\
=-\int_\RR\frac{|\tanh(\eta/2)|^{-\epsilon}}{\cosh^2(\eta/2)}\frac{(\epsilon-1)\cot(\feps(x))}{\cot^2(\feps(x))+\tanh^{2-2\epsilon}(\eta/2)}d\eta.
\end{multline*}
By notational convenience we use the notation $\sigma=\frac{\pi}{2}-\feps(x_t)$ and we define
$$\Pi^\epsilon=\frac{\tan(\theta)}{\tanh^{2-2\epsilon}(\eta/2)+\tan^2(\theta)}
+\frac{\tan(\sigma)}{\tan^2(\sigma)+\tanh^{2-2\epsilon}(\eta/2)}-\frac{\cot(\bar{\theta})}{\tanh^{2-2\epsilon}(\eta/2)+\cot^2(\bar{\theta})}
$$

Evaluating \eqref{IVeqreg2} in $x_t$ we have
\begin{eqnarray*}
\pat f^\epsilon(x_t)&=&-\sqrt{\epsilon} \alpha_1 \feps(x_t)+\epsilon\pax^2\feps(x_t)-\epsilon \alpha_2\Lambda_{l}^{1-\epsilon}\feps(x_t)-\epsilon \alpha_3\Lambda_{l}^{1-3\epsilon}\feps(x_t)\\
&&-(\Lambda_l-\Lambda_l^{1-\epsilon})\feps(x_t)
-\alpha_4(\Lambda_l-\Lambda_l^{1-3\epsilon})\feps(x_t)\nonumber\\\nonumber
&&-(1-\epsilon)\text{P.V.}\int_\RR\frac{|\tanh(\eta/2)|^{-\epsilon}}{\cosh^2(\eta/2)}\Pi^\epsilon d\eta.
\end{eqnarray*}
Using the definition of $\bar{\theta}$ and classical trigonometric identities we have
$$
\cot(\bar{\theta})=\tan\left(\frac{\pi}{2}-\bar{\theta}\right)=\tan\left(\frac{\pi}{2}-f(x_t)+\theta\right)=\frac{\tan(\frac{\pi}{2}-f(x_t))+\tan(\theta)}{1-\tan(\frac{\pi}{2}-f(x_t))\tan(\theta)}.
$$
Putting together all the terms in $\Pi^\epsilon$, we obtain
\begin{multline*}
\Pi^\epsilon=\frac{\tan(\sigma)\tan^2(\theta)[1+\tan^2(\sigma)|\tanh|^{2-2\epsilon}\left(\frac{\eta}{2}\right)][\tan^2(\sigma)+|\tanh|^{2-2\epsilon}\left(\frac{\eta}{2}\right)]^{-1}}{[(\tan(\sigma)+\tan(\theta))^2+(1-\tan(\sigma)\tan(\theta))^2|\tanh|^{2-2\epsilon}\left(\frac{\eta}{2}\right)]}\\
+\frac{2\tan^2(\sigma)\tan(\theta)[1-|\tanh|^{2-2\epsilon}\left(\frac{\eta}{2}\right)]}{[\tan^2(\sigma)+|\tanh|^{2-2\epsilon}\left(\frac{\eta}{2}\right)][(\tan(\sigma)+\tan(\theta))^2+(1-\tan(\sigma)\tan(\theta))^2|\tanh|^{2-2\epsilon}\left(\frac{\eta}{2}\right)]}\\
+\frac{\tan^2(\sigma)\tan(\theta)[1+\tan^2(\theta)|\tanh|^{2-2\epsilon}\left(\frac{\eta}{2}\right)]}{[\tan^2(\theta)+|\tanh|^{2-2\epsilon}\left(\frac{\eta}{2}\right)][(\tan(\sigma)+\tan(\theta))^2+(1-\tan(\sigma)\tan(\theta))^2|\tanh|^{2-2\epsilon}\left(\frac{\eta}{2}\right)]}\\
+\frac{2\tan(\sigma)\tan^2(\theta)[1-|\tanh|^{2-2\epsilon}\left(\frac{\eta}{2}\right)]}{[\tan^2(\theta)+|\tanh|^{2-2\epsilon}\left(\frac{\eta}{2}\right)][(\tan(\sigma)+\tan(\theta))^2+(1-\tan(\sigma)\tan(\theta))^2|\tanh|^{2-2\epsilon}\left(\frac{\eta}{2}\right)]}\\
+\frac{(\tan(\sigma)+\tan(\theta))\tan(\sigma)\tan(\theta)}{(\tan(\sigma)+\tan(\theta))^2+(1-\tan(\sigma)\tan(\theta))^2|\tanh|^{2-2\epsilon}\left(\frac{\eta}{2}\right)}.
\end{multline*}
Assuming that $0<\feps(x_t)=\max_x \feps(x)$, then $0<\tan(\theta),\tan(\sigma)$ and we obtain $\Pi^\epsilon\geq0$ and $\pat\feps(x_t)\leq0$. In the case $\feps(x_t)=\min_x \feps(x)<0$, we have $0>\tan(\theta),\tan(\sigma)$ and we get $\Pi^\epsilon\leq0$ and $\pat\feps(x_t)\geq0$. Integrating this in time, we get
$$
\|\feps(t)\|_{L^\infty(\RR)}\leq \|\feps_0\|_{L^\infty(\RR)}\leq\|f_0\|_{L^\infty(\RR)},
$$
where in the last step we use the definition \eqref{IVinitial}. In order to prove that the initial sign propagates we observe that if $f_0$ is positive (respectively negative) the same remains valid for $\feps_0$. Assume now that $f_0\geq0$ and suppose that the line $y=0$ is reached (if this line is not reached at any time $t$ we are done). We write $\feps(x_t)=\min_x \feps(x,t)=0$. We have $\tan(\theta)<0$, $\sigma=\pi/2$ and we get $\Pi^\epsilon\leq0$ and $\pat \feps(x_t)\geq0$. If $f_0\leq0$ we denote $\feps(x_t)=\max_x \feps(x,t)=0$. We have $\tan(\theta)>0$ and $\Pi^\epsilon\geq0$. Integrating in time we conclude the result.
\end{proof}

\subsection{Maximum principle for $\pax \feps$}
In this section we prove an \emph{a priori} bound for $\pax\feps$. We define 
$$
\mu_1(t)=\frac{\tan\left(\theta\right)}{\tanh\left(\frac{x_t-\eta}{2}\right)},\qquad \mu_2(t)=\tan\left(\bar{\theta}\right)\tanh\left(\frac{x_t-\eta}{2}\right)
$$ 
where $\theta$ and $\bar{\theta}$ are defined in \eqref{IVtheta} and $x_t$ is a critical point for $\pax \feps(x)$. We will use some bounds for $\mu_1$ and, for the reader's convenience, we collect them in the following lemma:
\begin{lem}\label{IVlemmu}
Let $f_0$ be an initial datum that fulfills \eqref{IVH3}, \eqref{IVH4} and \eqref{IVH5} (or \eqref{IVsisder2}), and let $\feps$ be the solution with initial datum $\feps_0$ defined in \eqref{IVinitial}. Then for $\mu_1$ the following inequalities hold
\begin{enumerate}
\item If $|x_t-\eta|\geq1$, due to \eqref{IVH4}, we have
\begin{equation}\label{IVfact2} |\mu_1(t)|\leq \frac{\tan\left(\|\feps (t)\|_{L^\infty}\right)}{\tanh\left(\frac{1}{2}\right)}\leq \frac{\tan\left(\|f_0\|_{L^\infty}\right)}{\tanh\left(\frac{1}{2}\right)}<\|\pax \feps_0\|_{L^ \infty}< 1.
\end{equation}
\item If $|x_t-\eta|\leq1$, we get
\begin{equation}\label{IVmu1}
|\mu_1(t)|\leq c\left(\|f_0\|^2_{L^\infty(\RR)}+1\right)\|\pax \feps(t)\|_{L^\infty(\RR)}.
\end{equation}
\item If $|x_t-\eta|\leq1$ and $x_t$ is the point where $\pax\feps$ reaches its maximum,
\begin{equation}\label{IVfact3}
\mu_1(t)-\pax \feps(x_t)\leq \frac{(x_t-\eta)^2}{48\tanh\left(\frac{1}{2}\right)}\left(|\pax \feps(x_t)|+5|\pax \feps(x_t)|^3\right).
\end{equation}
\item If $|x_t-\eta|\leq1$ and $\mu_1(t)-\pax \feps (x_t)\geq0$
\begin{multline}\label{IVfact4}
0\leq\mu_1^2(t)-(\pax \feps(x_t))^2\\
\leq\frac{(x_t-\eta)^2}{48\tanh\left(\frac{1}{2}\right)}\left(|\pax \feps(x_t)|+5|\pax \feps(x_t)|^3\right)\left(|\pax \feps(x_t)|+\frac{\tan\left(\frac{|\pax \feps(x_t)|}{2}\right)}{\tanh\left(\frac{1}{2}\right)}\right).
\end{multline}
\end{enumerate}
\end{lem}
\begin{proof}
To prove this lemma we use the following splitting
$$
\frac{\tan(\theta)}{\tanh((x_t-\eta)/2)}=\frac{\tan(\theta)-\theta}{\tanh((x_t-\eta)/2)}\\
+\frac{\theta}{\tanh((x_t-\eta)/2)},
$$
Taylor's theorem and the appropriate bounds using Proposition \ref{IVMPf}.
\end{proof}

First, we assume $\pax \feps(x_t)=\max_x \pax \feps(x,t)$. Notice that we can take $0<\epsilon<1/10$ small enough to ensure that $\feps(x,0)$ defined in \eqref{IVinitial} also fulfills the hypotheses \eqref{IVH3}, \eqref{IVH4} and \eqref{IVH5}. From \eqref{IVeqreg2}, taking one derivative and using Lemma \ref{IVops}, we get
\begin{eqnarray}
\pat \pax \feps(x_t)&=&-8\pax \feps(x_t)-\sqrt{\epsilon} \alpha_1\pax \feps(x_t)+\epsilon\pax^3\feps(x_t)-\epsilon \alpha_2\Lambda_{l}^{1-\epsilon}\feps(x_t) \nonumber\\
&&-\epsilon \alpha_3\Lambda_{l}^{1-3\epsilon}\feps(x_t)-\alpha_4(\Lambda_l-\Lambda_l^{1-3\epsilon})\feps(x_t)
-(\Lambda_l-\Lambda_l^{1-\epsilon})\feps(x_t) \label{IVeqreg31}\\
&&+\text{P.V.}\int_\RR \mathcal{I}_1d\eta+\text{P.V.}\int_\RR \mathcal{I}_2d\eta+\text{P.V.}\int_\RR \mathcal{I}_3d\eta\label{IVeqreg35}
\end{eqnarray}
where $\mathcal{I}_1$ is the integral corresponding to $\Xi_1^\epsilon$, $\mathcal{I}_2$ is the integral corresponding to $\Xi_2^\epsilon$ and
\begin{eqnarray}
\mathcal{I}_3&=&+\epsilon\text{P.V.}\int_\RR\frac{|\tanh(\eta/2)|^\epsilon\pax\theta\left(1-|\tanh(\eta/2)|^{2\epsilon}\mu_1^2(t)\right)d\eta}{\sinh^2(\eta/2)\cos^2(\theta)\left(1+\frac{\tan^2(\theta)|\tanh\left(\eta/2\right)|^{2\epsilon}}{\tanh^2\left(\eta/2\right)}\right)^2}\label{IVeqreg32}\\
&&-\epsilon\text{P.V.}\int_\RR\frac{|\tanh(\eta/2)|^{-\epsilon}\pax\bar{\theta}d\eta}{\cosh^2(\eta/2)\cos^2(\bar{\theta})\left(1+\frac{\tan^2(\bar{\theta})\tanh^2\left(\eta/2\right)|}{|\tanh\left(\eta/2\right)|^{2\epsilon}}\right)} \label{IVeqreg33}\\
&&+\epsilon\text{P.V.}\int_\RR\frac{\mu_2^2(t)|\tanh((x_t-\eta)/2)|^{-3\epsilon}2\pax\bar{\theta}d\eta}{\cosh^2((x_t-\eta)/2)\cos^2(\bar{\theta})\left(1+\frac{\tan^2(\bar{\theta})\tanh^2\left((x-\eta)/2\right)|}{|\tanh\left((x-\eta)/2\right)|^{2\epsilon}}\right)^2}. \label{IVeqreg34}
\end{eqnarray}
This extra term appear from the regularization present in both $\Xi_i^\epsilon$.

We have
$$
\mathcal{I}_1=\Gamma_1+\epsilon\Gamma_2
$$
where
\begin{multline}\label{IVeqGamma1}
\Gamma_1=\frac{\left(|\tanh((x_t-\eta)/2)|^{3\epsilon}-1\right)\Gamma_1^1}{\left(1+\frac{\tan^2(\theta)|\tanh\left((x-\eta)/2\right)|^{2\epsilon}}{\tanh^2\left((x-\eta)/2\right)}\right)^2}
+\frac{\left(|\tanh((x_t-\eta)/2)|^\epsilon-1\right)\Gamma_1^2}{\left(1+\frac{\tan^2(\theta)|\tanh\left((x-\eta)/2\right)|^{2\epsilon}}{\tanh^2\left((x-\eta)/2\right)}\right)^2}\\
+\frac{\Gamma_1^1+\Gamma_1^2}{\left(1+\frac{\tan^2(\theta)|\tanh\left((x-\eta)/2\right)|^{2\epsilon}}{\tanh^2\left((x-\eta)/2\right)}\right)^2},
\end{multline}
with
$$
\Gamma_1^1=\frac{-(\pax\feps(x_t))^2\mu_1}{\cos^2(\theta)\tanh^2((x_t-\eta)/2)}+\frac{\mu_1^3}{\cosh^2((x_t-\eta)/2)}+\frac{\pax \feps(x_t)\mu_1^2}{\sinh^2((x_t-\eta)/2)\cos^2(\theta)},
$$
and
$$
\Gamma_1^2=\frac{\mu_1-\pax\feps(x_t)}{\sinh^2((x_t-\eta)/2)}-\frac{\pax\feps(x_t)\mu_1^2}{\cosh^2((x_t-\eta)/2)}+\frac{(\pax \feps(x_t))^2\mu_1}{\cos^2(\theta)}.
$$
The second term is given by
\begin{equation}\label{IVeqGamma2}
\Gamma_2=\frac{|\tanh((x_t-\eta)/2)|^{3\epsilon}\Gamma_2^1}{2\left(1+\frac{\tan^2(\theta)|\tanh\left((x-\eta)/2\right)|^{2\epsilon}}{\tanh^2\left((x-\eta)/2\right)}\right)^2}
+\frac{|\tanh((x_t-\eta)/2)|^\epsilon\Gamma_2^2}{2\left(1+\frac{\tan^2(\theta)|\tanh\left((x-\eta)/2\right)|^{2\epsilon}}{\tanh^2\left((x-\eta)/2\right)}\right)^2},
\end{equation}
where
$$
\Gamma_2^1=\mu_1^2(t)\frac{\frac{\pax\feps(x_t)}{\cos^2(\theta)}+\frac{-\mu_1(t)}{\cosh^2((x_t-\eta)/2)}}{\sinh^2((x_t-\eta)/2)},\text{ and } \Gamma_2^2=\frac{\frac{\pax\feps(x_t)}{\cos^2(\theta)}+\frac{-\mu_1(t)}{\cosh^2((x_t-\eta)/2)}}{\sinh^2((x_t-\eta)/2)}.
$$
We compute 
$$
\mathcal{I}_2=\Omega_1+\epsilon\Omega_2,
$$
with 
\begin{multline}\label{IVeqOmega1}
\Omega_1=\frac{\left(|\tanh((x_t-\eta)/2)|^{-3\epsilon}-1\right)\Omega_1^1}{\left(1+\frac{\tan^2(\bar{\theta})\tanh^2\left((x_t-\eta)/2\right)}{|\tanh\left((x_t-\eta)/2\right)|^{2\epsilon}}\right)^2}
+\frac{\left(|\tanh((x_t-\eta)/2)|^{-\epsilon}-1\right)\Omega_1^2}{\left(1+\frac{\tan^2(\bar{\theta})\tanh^2\left((x_t-\eta)/2\right)}{|\tanh\left((x_t-\eta)/2\right)|^{2\epsilon}}\right)^2}\\
+\frac{\Omega_1^1+\Omega_1^2}{\left(1+\frac{\tan^2(\bar{\theta})\tanh^2\left((x_t-\eta)/2\right)}{|\tanh\left((x_t-\eta)/2\right)|^{2\epsilon}}\right)^2},
\end{multline}
where
\begin{multline*}
\Omega_1^1=-\frac{\pax\feps(x_t)\mu_2^2(t)\sec^2(\bar{\theta})}{\cosh^2((x_t-\eta)/2)}+(\pax\feps(x_t))^2\mu_2^3(t)\sec^2(\bar{\theta})-
\frac{\mu_2^3(t)}{\cosh^2((x_t-\eta)/2)}\\
-(\pax\feps(x_t))^2\mu_2(t)\tanh^2((x_t-\eta)/2)\sec^4(\bar{\theta})-\frac{\tan^ 2(\bar{\theta})\mu_2(t)}{\cosh^4((x_t-\eta)/2)},
\end{multline*}
and
$$
\Omega_1^2=\frac{\pax\feps(x_t)\sec^2(\bar{\theta})-\mu_2(t)}{\cosh^2((x_t-\eta)/2)}+(\pax\feps(x_t))^2\mu_2(t)\sec^ 2(\bar{\theta}).
$$
The second term is given by
\begin{multline}\label{IVeqOmega2}
\Omega_2=\frac{|\tanh((x_t-\eta)/2)|^{-3\epsilon}\left(\frac{\pax\feps(x_t)\mu_2^2(t)\sec^ 2(\bar{\theta})}{2\cosh^2((x_t-\eta)/2)}+\frac{\tan^2(\bar{\theta})\mu_2(t)}{2\cosh^4((x_t-\eta)/2)}\right)}{\left(1+\frac{\tan^2(\bar{\theta})\tanh^2\left((x_t-\eta)/2\right)}{|\tanh\left((x_t-\eta)/2\right)|^{2\epsilon}}\right)^2}\\
+\frac{|\tanh((x_t-\eta)/2)|^{-\epsilon}\left(\frac{-\pax\feps(x_t)\sec^ 2(\bar{\theta})}{2\cosh^2((x_t-\eta)/2)}-\frac{\tan(\bar{\theta})}{2\cosh^4((x_t-\eta)/2)\tanh((x_t-\eta)/2)}\right)}{\left(1+\frac{\tan^2(\bar{\theta})\tanh^2\left((x_t-\eta)/2\right)}{|\tanh\left((x_t-\eta)/2\right)|^{2\epsilon}}\right)^2}.
\end{multline}

We need to obtain the local decay $\|\pax\feps(t)\|_{L^\infty(\RR)}\leq\|\pax\feps(0)\|_{L^\infty(\RR)}$ for $0\leq t<t^*$. Assuming the classical solvability for \eqref{IVeqreg} with an initial datum $f_0$ fulfilling the hypotheses \eqref{IVH3},  \eqref{IVH4} and \eqref{IVH5} we have that $\feps(x,\delta)$ also fulfills \eqref{IVH3}, \eqref{IVH4} and \eqref{IVH5} if $0\leq\delta<<1$ is small enough.  Recall that $\pax \feps (x_\delta)=\|\pax \feps (\delta)\|_{L^\infty(\RR)}$ and $\pax\theta>0$. The linear terms in \eqref{IVeqreg31} have the appropriate sign and they will be used to control the the positive contributions of the nonlinear terms. We need to prove that $\pat\pax \feps(x_\delta)<0$. For the sake of simplicity, we split the proof of this inequality in different lemmas.

\begin{lem}\label{IVlemdf1}
If $\alpha_2>2\sec^2(\|f_0\|_{L^\infty(\RR)})$, we have
$$
\mathcal{I}_3\leq \epsilon c\tan^ 2\left(\|f_0\|_{L^ \infty(\RR)}\right)\sec^ 2\left(\|f_0\|_{L^ \infty(\RR)}\right)\pax \feps(x_\delta)
$$
\end{lem}
\begin{proof}
Using the linear term $\Lambda^{1-\epsilon}_l$ to control \eqref{IVeqreg32}, we have
\begin{multline*}
A_1=\epsilon\text{P.V.}\int_\RR\frac{|\tanh(\eta/2)|^\epsilon\pax\theta\left(1-|\tanh(\eta/2)|^{2\epsilon}\mu_1^2(\delta)\right)d\eta}{\sinh^2(\eta/2)\cos^2(\theta)\left(1+\frac{\tan^2(\theta)|\tanh\left(\eta/2\right)|^{2\epsilon}}{\tanh^2\left(\eta/2\right)}\right)^2}
-\epsilon \frac{\alpha_2}{2}\Lambda_l^{1-\epsilon}\pax\feps(x_\delta)\\
=\epsilon\text{P.V.}\int_\RR\frac{|\tanh(\eta/2)|^\epsilon\pax\theta\left(\frac{1}{\cos^2(\theta)\left(1+\frac{\tan^2(\theta)|\tanh\left(\eta/2\right)|^{2\epsilon}}{\tanh^2\left(\eta/2\right)}\right)^2}-\frac{\alpha_2}{2}\right)d\eta}{\sinh^2(\eta/2)}\\
-\epsilon\text{P.V.}\int_\RR\frac{|\tanh(\eta/2)|^\epsilon\pax\theta|\tanh(\eta/2)|^{2\epsilon}\mu_1^2(\delta)d\eta}{\sinh^2(\eta/2)\cos^2(\theta)\left(1+\frac{\tan^2(\theta)|\tanh\left(\eta/2\right)|^{2\epsilon}}{\tanh^2\left(\eta/2\right)}\right)^2}<0,
\end{multline*}
if $\alpha_2/2>\sec^2(\|f_0\|_{L^\infty(\RR)})$. Due to $\pax\feps(x_\delta)=\|\pax\feps(\delta)\|_{L^\infty(\RR)}$, we have $\pax\bar{\theta}>0$. Then, the term \eqref{IVeqreg33} is
\begin{equation*}
A_2=-\epsilon\text{P.V.}\int_\RR\frac{|\tanh(\eta/2)|^{-\epsilon}\pax\bar{\theta}d\eta}{\cosh^2(\eta/2)\cos^2(\bar{\theta})\left(1+\frac{\tan^2(\bar{\theta})\tanh^2\left(\eta/2\right)}{|\tanh\left(\eta/2\right)|^{2\epsilon}}\right)} <0.
\end{equation*}
The term \eqref{IVeqreg34} is
\begin{multline*}
A_3=\epsilon\text{P.V.}\int_\RR\frac{\mu_2^2(\delta)|\tanh((x_\delta-\eta)/2)|^{-3\epsilon}2\pax\bar{\theta}d\eta}{\cosh^2((x_\delta-\eta)/2)\cos^2(\bar{\theta})\left(1+\frac{\tan^2(\bar{\theta})\tanh^2\left((x_\delta-\eta)/2\right)}{|\tanh\left((x_\delta-\eta)/2\right)|^{2\epsilon}}\right)^2}\\
\leq\epsilon c\tan^ 2\left(\|f_0\|_{L^ \infty(\RR)}\right)\sec^ 2\left(\|f_0\|_{L^ \infty(\RR)}\right)\pax \feps(x_\delta).
\end{multline*}
\end{proof}

This kind of terms will be absorbed by $\alpha_1$. We have to deal with $\mathcal{I}_1$. We start with the term corresponding to $\Gamma^2_2$ in \eqref{IVeqGamma2}. We write
\begin{equation*}
A_4=\epsilon\text{P.V.}\int_\RR\frac{|\tanh((x_\delta-\eta)/2)|^\epsilon\left(\frac{\pax\feps(x_\delta)}{\cos^2(\theta)}+\frac{-\mu_1(\delta)}{\cosh^2((x_\delta-\eta)/2)}\right)}{2\sinh^2((x_\delta-\eta)/2)\left(1+\frac{\tan^2(\theta)|\tanh\left((x_\delta-\eta)/2\right)|^{2\epsilon}}{\tanh^2\left((x_\delta-\eta)/2\right)}\right)^2}d\eta.
\end{equation*}
\begin{lem}\label{IVlemdf2}
If $\alpha_2>2\sec^2(\|f_0\|_{L^\infty(\RR)})$, we have
$$
A_4\leq c\epsilon\pax f(x_\delta)\left(\sec\left(\|f_0\|_{L^\infty(\RR)}\right)+1\right)+c\epsilon\pax f(x_\delta)\frac{\alpha_2}{2}.
$$
\end{lem}
\begin{proof}
We split 
\begin{multline*}
A_4=\epsilon\text{P.V.}\int_\RR\frac{|\tanh((x_\delta-\eta)/2)|^\epsilon\left(\frac{\pax\feps(x_\delta)}{\cos^2(\theta)}-\pax\feps(x_\delta)+\frac{-\mu_1(\delta)}{\cosh^2((x_\delta-\eta)/2)}+\mu_1(t)\right)}{2\sinh^2((x_\delta-\eta)/2)\left(1+\frac{\tan^2(\theta)|\tanh\left((x_\delta-\eta)/2\right)|^{2\epsilon}}{\tanh^2\left((x_\delta-\eta)/2\right)}\right)^2}d\eta\\
+\epsilon\text{P.V.}\int_\RR\frac{|\tanh((x_\delta-\eta)/2)|^\epsilon\left(\pax\feps(x_\delta)-\mu_1(\delta)\right)}{2\sinh^2((x_\delta-\eta)/2)\left(1+\frac{\tan^2(\theta)|\tanh\left((x_\delta-\eta)/2\right)|^{2\epsilon}}{\tanh^2\left((x_\delta-\eta)/2\right)}\right)^2}d\eta=B_1+B_2.
\end{multline*}
Since $0<\delta<<1$ is small enough to ensure that the hypotheses \eqref{IVH3},  \eqref{IVH4} and \eqref{IVH5} hold at time $\delta$, we have that, if $|\eta|>1$,
\begin{equation}\label{IVpropmu1delta}
|\mu_1(\delta)|\leq\frac{\tan(\|\feps(\delta)\|_{L^\infty(\RR)})}{\tanh(1/2)}<\pax\feps(x_\delta),
\end{equation}
The term $B_1$ is not singular and can be bounded using \eqref{IVmu1} and \eqref{IVpropmu1delta}:
\begin{multline*}
|B_1|\leq \epsilon\text{P.V.}\left(\int_{B(0,1)}+\int_{B^ c(0,1)}\right)\frac{\pax\feps(x_\delta)\tan^2(\theta)+|\mu_1(\delta)|\tanh^2(\eta/2)}{2\sinh^2(\eta/2)}d\eta\\
\leq c\epsilon\pax f(x_\delta)\left(\sec\left(\|f_0\|_{L^\infty(\RR)}\right)+1\right).
\end{multline*}
We compute
$$
B_2=\epsilon\text{P.V.}\int_{\RR}\frac{|\tanh(\eta/2)|^\epsilon\left(\pax\feps(x_\delta)-\frac{\tan(\theta)-\theta}{\tanh(\eta/2)}-\frac{\theta}{\tanh(\eta/2)}+\frac{2\theta}{\eta}-\frac{2\theta}{\eta}\right)}{2\sinh^2(\eta/2)\left(1+\frac{\tan^2(\theta)|\tanh\left(\eta/2\right)|^{2\epsilon}}{\tanh^2\left(\eta/2\right)}\right)^2}d\eta=C_1+C_2,
$$
with 
$$
C_1=\epsilon\text{P.V.}\left(\int_{B(0,1)}+\int_{B^c(0,1)}\right)\frac{|\tanh(\eta/2)|^\epsilon\left(-\frac{\tan(\theta)-\theta}{\tanh(\eta/2)}-\frac{\theta}{\tanh(\eta/2)}+\frac{2\theta}{\eta}\right)}{2\sinh^2(\eta/2)\left(1+\frac{\tan^2(\theta)|\tanh\left(\eta/2\right)|^{2\epsilon}}{\tanh^2\left(\eta/2\right)}\right)^2}d\eta=D_1+D_2,
$$
and
$$
C_2=\epsilon\text{P.V.}\int_{\RR}\frac{|\tanh(\eta/2)|^\epsilon\left(\pax\feps(x_\delta)-\frac{2\theta}{\eta}\right)}{2\sinh^2(\eta/2)\left(1+\frac{\tan^2(\theta)|\tanh\left(\eta/2\right)|^{2\epsilon}}{\tanh^2\left(\eta/2\right)}\right)^2}d\eta.
$$
Using the Mean Value Theorem, we bound the inner term $D_1$ as
$$
|D_1|\leq c\epsilon\pax \feps(x_\delta).
$$
Due to \eqref{IVpropmu1delta}, the outer term is
$$
|D_2|\leq\epsilon\text{P.V.}\int_{B^c(0,1)}\frac{|\mu_1(\delta)|+\pax\feps(x_\delta)}{2\sinh^2(\eta/2)}d\eta\leq \epsilon c\pax\feps(x_\delta).
$$
Putting all together, we obtain 
\begin{equation*}
|C_1|\leq \epsilon c\pax\feps(x_\delta).
\end{equation*}
Then, using the diffusion given by $\Lambda_l^{1-\epsilon}$ to control $C_2$, we get
\begin{multline*}
C_2-\epsilon\frac{\alpha_2}{2}\Lambda_l^{1-\epsilon}\pax\feps(x_\delta)=\epsilon\text{P.V.}\int_{\RR}\frac{|\tanh(\eta/2)|^\epsilon\left(\pax\feps(x_\delta)-\frac{2\theta}{\eta}\right)}{2\sinh^2(\eta/2)\left(1+\frac{\tan^2(\theta)|\tanh\left(\eta/2\right)|^{2\epsilon}}{\tanh^2\left(\eta/2\right)}\right)^2}d\eta\\
-\epsilon\frac{\alpha_2}{2}\left((1-\epsilon)\text{PV}\int_\RR\frac{\left(\pax \feps(x_\delta)-\frac{\feps(x_\delta)-\feps(x_\delta-\eta)}{\sinh(\eta)}\right)|\tanh(\eta/2)|^\epsilon}{\sinh^2\left(\frac{\eta}{2}\right)}d\eta\right.\\
\left.+\text{P.V.}\int_\RR\frac{\left(\pax \feps(x_\delta)-\frac{\feps(x_\delta)-\feps(x_\delta-\eta)}{\tanh(\eta)}\right)|\tanh(\eta/2)|^\epsilon}{\sinh^2\left(\frac{\eta}{2}\right)}d\eta+4\pax \feps(x_\delta)\right).
\end{multline*}
Due to $|\eta/\sinh(\eta)|<1$ and $0<\epsilon<1/10$, some terms have the appropriate sign:
$$
\epsilon\frac{\alpha_2}{2}\left((1-\epsilon)\text{PV}\int_\RR\frac{\left(\pax \feps(x_\delta)-\frac{\feps(x_\delta)-\feps(x_\delta-\eta)}{\sinh(\eta)}\right)|\tanh(\eta/2)|^\epsilon}{\sinh^2\left(\frac{\eta}{2}\right)}d\eta+4\pax \feps(x_\delta)\right)\geq0,
$$
thus we can neglect their contribution. Furthermore, we have
\begin{eqnarray*}
C_2-\epsilon\frac{\alpha_2}{2}\Lambda_l^{1-\epsilon}\pax\feps(x_\delta)&<&\epsilon\text{P.V.}\int_{\RR}\frac{|\tanh(\eta/2)|^\epsilon\left(\pax\feps(x_\delta)-\frac{2\theta}{\eta}\right)}{2\sinh^2(\eta/2)\left(1+\frac{\tan^2(\theta)|\tanh\left(\eta/2\right)|^{2\epsilon}}{\tanh^2\left(\eta/2\right)}\right)^2}d\eta\\
&&-\epsilon\frac{\alpha_2}{2}\text{PV}\int_\RR\frac{\left(\pax \feps(x_\delta)-\frac{2\theta}{\eta}+\frac{2\theta}{\eta}-\frac{2\theta}{\tanh(\eta)}\right)|\tanh(\eta/2)|^\epsilon}{\sinh^2\left(\frac{\eta}{2}\right)}d\eta\\
&\leq & \epsilon\text{P.V.}\int_{\RR}\frac{|\tanh(\eta/2)^\epsilon\left(\pax\feps(x_\delta)-\frac{2\theta}{\eta}\right)}{\sinh^2(\eta/2)}\\
&&\cdot\left(\frac{1}{2\left(1+\frac{\tan^2(\theta)|\tanh\left(\eta/2\right)|^{2\epsilon}}{\tanh^2\left(\eta/2\right)}\right)^2}-\frac{\alpha_2}{2}\right)d\eta\\
&&-\epsilon\frac{\alpha_2}{2}\text{PV}\int_\RR\frac{\left(\frac{2\theta}{\eta}-\frac{2\theta}{\tanh(\eta)}\right)|\tanh(\eta/2)|^\epsilon}{\sinh^2\left(\frac{\eta}{2}\right)}d\eta.
\end{eqnarray*}
Taking $\alpha_2/2>1$ and using the Mean Value Theorem, we get
\begin{equation*}
C_2-\epsilon\frac{\alpha_2}{2}\Lambda_l^{1-\epsilon}\pax\feps(x_\delta)<\epsilon\frac{\alpha_2}{2}\pax \feps(x_\delta)c.
\end{equation*}
Combining these terms we conclude this result.
\end{proof}

The term corresponding to $\Gamma^1_2$ in \eqref{IVeqGamma2} is
\begin{equation*}
A_5=\epsilon\text{P.V.}\int_\RR\frac{|\tanh(\eta/2)|^{3\epsilon}\mu_1^2(\delta)\left(\frac{\pax\feps(x_\delta)}{\cos^2(\theta)}+\frac{-\mu_1(\delta)}{\cosh^2(\eta/2)}\right)}{2\sinh^2(\eta/2)\left(1+\frac{\tan^2(\theta)|\tanh\left(\eta/2\right)|^{2\epsilon}}{\tanh^2\left(\eta/2\right)}\right)^2}d\eta.
\end{equation*}
\begin{lem}\label{IVlemdf3}
If $\alpha_3>1$, we have
$$
A_5\leq c\epsilon\pax f(x_\delta)\left(\sec\left(\|f_0\|_{L^\infty(\RR)}\right)+1\right)+c\epsilon\pax f(x_\delta)\alpha_3.
$$
\end{lem}
\begin{proof}
The proof follows the same ideas as in Lemma \ref{IVlemdf2}.
\end{proof}

We are done with $\Gamma^1_2$, thus, using the previous bound for $\Gamma_2^2$, we are done with $\Gamma_2$ in \eqref{IVeqGamma2}. The terms in $\Gamma_1$ are not multiplied by $\epsilon$ and we have to obtain this decay from the integral. We write
$$
A_6=\text{P.V.}\int_{\RR}\frac{\left(|\tanh(\eta/2)|^\epsilon-1\right)\Gamma_1^2}{\left(1+\frac{\tan^2(\theta)|\tanh\left(\eta/2\right)|^{2\epsilon}}{\tanh^2\left(\eta/2\right)}\right)^2}.
$$
\begin{lem}\label{IVlemdf4}
We have
$$
A_6\leq c\epsilon\pax f(x_\delta)\left(\sec^2\left(\|f_0\|_{L^\infty(\RR)}\right)+1\right).
$$
\end{lem}
\begin{proof}
We have
$$
A_6=B_5+B_6+B_7,
$$
with
$$
B_5=\text{P.V.}\left(\int_{B(0,\epsilon)}+\int_{B^c(0,\epsilon)\cap B(0,1)}+\int_{B^c(0,1)}\right)\frac{\left(|\tanh(\eta/2)|^\epsilon-1\right)\left(-\pax\feps(x_\delta)\mu_1^2(\delta)\right)}{\cosh^2(\eta/2)\left(1+\frac{\tan^2(\theta)|\tanh\left(\eta/2\right)|^{2\epsilon}}{\tanh^2\left(\eta/2\right)}\right)^2}d\eta,
$$
$$
B_6=\text{P.V.}\left(\int_{B(0,\epsilon)}+\int_{B^c(0,\epsilon)\cap B(0,1)}+\int_{B^c(0,1)}\right)\frac{\left(|\tanh(\eta/2)|^\epsilon-1\right)(\pax \feps(x_\delta))^2\mu_1(\delta)}{\cos^2(\theta)\left(1+\frac{\tan^2(\theta)|\tanh\left(\eta/2\right)|^{2\epsilon}}{\tanh^2\left(\eta/2\right)}\right)^2}d\eta,
$$
and
$$
B_7=\text{P.V.}\int_{\RR}\frac{\left(|\tanh(\eta/2)|^\epsilon-1\right)\left(\mu_1(\delta)-\pax\feps(x_\delta)\right)}{\sinh^2(\eta/2)\left(1+\frac{\tan^2(\theta)|\tanh\left(\eta/2\right)|^{2\epsilon}}{\tanh^2\left(\eta/2\right)}\right)^2}d\eta.
$$
The term $B_5$ is not singular and can be bounded using \eqref{IVfact1} and \eqref{IVfact1.b} as follows:
\begin{multline*}
|B_5|\leq 4\epsilon\pax\feps(x_\delta)+\epsilon\int_{B(0,1)}|\log\left(|\tanh(\eta/2)|\right)|\pax\feps(x_\delta)d\eta\\
+\epsilon\int_{B^c(0,1)}\frac{|\log\left(|\tanh(\eta/2)|\right)|\pax\feps(x_\delta)}{\cosh^2(\eta/2)}d\eta\leq c\epsilon\pax\feps(x_\delta).
\end{multline*}
We can bound $B_6$ in the same way,
\begin{multline*}
|B_6|\leq 4\epsilon\sec^2\left(\|f_0\|_{L^\infty(\RR)}\right)\pax\feps(x_\delta)\\
+\epsilon \sec^2\left(\|f_0\|_{L^\infty(\RR)}\right)\int_{\RR}|\log\left(|\tanh(\eta/2)|\right)|\pax\feps(x_\delta)d\eta\\
\leq c\epsilon\sec^2\left(\|f_0\|_{L^\infty(\RR)}\right)\pax\feps(x_\delta).
\end{multline*}
We split the term $B_7$ as follows
$$
B_7=\text{P.V.}\int_{\RR}\frac{\left(|\tanh(\eta/2)|^\epsilon-1\right)\left(\frac{\tan(\theta)-\theta}{\tanh(\eta/2)}+\frac{\theta}{\tanh(\eta/2)}-\frac{2\theta}{\eta}+\frac{2\theta}{\eta}-\pax\feps(x_\delta)\right)}{\sinh^2(\eta/2)\left(1+\frac{\tan^2(\theta)|\tanh\left(\eta/2\right)|^{2\epsilon}}{\tanh^2\left(\eta/2\right)}\right)^2}d\eta=C_5+C_6,
$$
where
\begin{multline*}
C_5=\text{P.V.}\left(\int_{B(0,\epsilon)}+\int_{B^c(0,\epsilon)}\right)\frac{\left(|\tanh(\eta/2)|^\epsilon-1\right)\left(\frac{\tan(\theta)-\theta}{\tanh(\eta/2)}+\frac{\theta}{\tanh(\eta/2)}-\frac{2\theta}{\eta}\right)}{\sinh^2(\eta/2)\left(1+\frac{\tan^2(\theta)|\tanh\left(\eta/2\right)|^{2\epsilon}}{\tanh^2\left(\eta/2\right)}\right)^2}d\eta\\
\leq c\epsilon\pax\feps(x_\delta),
\end{multline*}
and
$$
C_6=\text{P.V.}\int_{\RR}\frac{\left(1-|\tanh(\eta/2)|^\epsilon\right)\left(\pax\feps(x_\delta)-\frac{2\theta}{\eta}\right)}{\sinh^2(\eta/2)\left(1+\frac{\tan^2(\theta)|\tanh\left(\eta/2\right)|^{2\epsilon}}{\tanh^2\left(\eta/2\right)}\right)^2}d\eta.
$$
To bound $C_6$ we need to use the diffusion coming from $\Lambda_l-\Lambda_l^{1-\epsilon}$. Notice that, according to Lemma \ref{IVops}, we have
\begin{multline*}
\left(\Lambda_l-\Lambda_l^{1-\epsilon}\right)\pax \phi(x)=(1-\epsilon)\text{PV}\int_\RR\frac{\left(\pax \phi(x)-\frac{\phi(x)-\phi(\eta)}{\sinh(x-\eta)}\right)\left(1-|\tanh((x-\eta)/2)|^\epsilon\right)}{\sinh^2\left(\frac{x-\eta}{2}\right)}d\eta\\
+\epsilon\text{PV}\int_\RR\frac{\left(\pax \phi(x)-\frac{\phi(x)-\phi(\eta)}{\sinh(x-\eta)}\right)}{\sinh^2\left(\frac{x-\eta}{2}\right)}d\eta\\
+\text{PV}\int_\RR\frac{\left(\pax \phi(x)-\frac{\phi(x)-\phi(\eta)}{\tanh(x-\eta)}\right)\left(1-|\tanh((x-\eta)/2)|^\epsilon\right)}{\sinh^2\left(\frac{x-\eta}{2}\right)}d\eta,
\end{multline*}
and, when evaluating in the point where $\pax \phi(x)$ reaches its maximum, the first two terms are positive and they can be neglected. We get
\begin{multline*}
C_6-\left(\Lambda_l-\Lambda_l^{1-\epsilon}\right)\pax \feps(x_\delta)<\text{P.V.}\int_{\RR}\frac{\left(1-|\tanh(\eta/2)|^\epsilon\right)\left(\pax\feps(x_\delta)-\frac{2\theta}{\eta}\right)}{\sinh^2(\eta/2)\left(1+\frac{\tan^2(\theta)|\tanh\left(\eta/2\right)|^{2\epsilon}}{\tanh^2\left(\eta/2\right)}\right)^2}\\
-\text{PV}\int_\RR\frac{\left(\pax \feps(x_\delta)-\frac{2\theta}{\eta}+\frac{2\theta}{\eta}-\frac{\feps(x_\delta)-\feps(\eta)}{\tanh(\eta)}\right)\left(1-|\tanh(\eta/2)|^\epsilon\right)}{\sinh^2\left(\frac{\eta}{2}\right)}d\eta\\
\leq \text{PV}\int_\RR\frac{\left(\pax \feps(x_\delta)-\frac{2\theta}{\eta}\right)\left(1-|\tanh(\eta/2)|^\epsilon\right)}{\sinh^2\left(\frac{\eta}{2}\right)}\left(\frac{1}{\left(1+\frac{\tan^2(\theta)|\tanh\left(\eta/2\right)|^{2\epsilon}}{\tanh^2\left(\eta/2\right)}\right)^2}-1\right)d\eta\\
-\text{PV}\int_\RR\frac{\left(\frac{2\theta}{\eta}-\frac{\feps(x_\delta)-\feps(\eta)}{\tanh(\eta)}\right)\left(1-|\tanh(\eta/2)|^\epsilon\right)}{\sinh^2\left(\frac{\eta}{2}\right)}d\eta\leq c\epsilon\pax\feps(x_\delta),
\end{multline*}
where in the last step we have used the previous splitting in $B(0,\epsilon)$ and $\RR-B(0,\epsilon)$, \eqref{IVfact1} and \eqref{IVfact1.b}. This concludes the result. 
\end{proof}

Now that we have finished with $\Gamma^2_1$, the term with $\Gamma_1^1$ is
$$
A_7=\text{P.V.}\int_{\RR}\frac{\left(|\tanh(\eta/2)|^{3\epsilon}-1\right)\Gamma_1^1}{\left(1+\frac{\tan^2(\theta)|\tanh\left(\eta/2\right)|^{2\epsilon}}{\tanh^2\left(\eta/2\right)}\right)^2}d\eta.
$$
We have
\begin{lem}\label{IVlemdf5}
If $\alpha_4>\sec^2\left(\|f_0\|_{L^\infty(\RR)}\right)$, we have
$$
A_7\leq c\epsilon\left(\sec\left(\|f_0\|_{L^\infty(\RR)}\right)+1\right)^7\pax\feps(x_\delta)+\alpha_4c\epsilon\pax\feps(x_\delta).
$$
\end{lem}
\begin{proof}
The proof is similar to the proof of Lemma \ref{IVlemdf4} and, for the sake of brevity, omit it. 
\end{proof}

In order to finish bounding $\Gamma_1$ in \eqref{IVeqGamma1}, we have to bound the term
$$
A_8=\text{P.V.}\int_{\RR}\frac{\Gamma_1^1+\Gamma_1^2}{\left(1+\frac{\tan^2(\theta)|\tanh\left(\eta/2\right)|^{2\epsilon}}{\tanh^2\left(\eta/2\right)}\right)^2}d\eta.
$$
This term, akin to the singular term in \cite{CGO}, is bounded using the hypotheses \eqref{IVH3} and \eqref{IVH4}.
\begin{lem}\label{IVlemdf6}
Using \eqref{IVH3}, \eqref{IVH4} and \eqref{IVH5}, we obtain
$$
A_8\leq \frac{\left(\pax \feps(x_\delta)+5(\pax \feps(x_\delta))^3\right)\left(1+\pax \feps(x_\delta)\left(\pax \feps(x_\delta)+\frac{\tan\left(\frac{\pax \feps(x_\delta)}{2}\right)}{\tanh\left(\frac{1}{2}\right)}\right)\right)}{6\tanh(1/2)\cos^2(\|\feps(\delta)\|_{L^\infty(\RR)})}.
$$
\end{lem}
\begin{proof}
Using classical trigonometric identities we can write
\begin{multline*}
\Gamma_1^1=\frac{\pax\feps(x_\delta)\sec^2(\theta)\mu_1^2(\delta)}{\sinh^2(\eta/2)}+(\pax\feps(x_\delta))^2\mu^3_1(\delta)\sec^2(\theta)\\
+\frac{\mu_1^3(\delta)}{\sinh^ 2(\eta/2)}-\frac{(\pax\feps(x_\delta))^2\sec^4(\theta)\mu_1(\delta)}{\tanh^2(\eta/2)}-\frac{\mu_1(\delta)\tan^2(\theta)}{\sinh^4(\eta/2)},
\end{multline*}
$$
\Gamma_1^2=-\frac{\pax\feps(x_\delta)\sec^2(\theta)}{\sinh^2(\eta/2)}+(\pax\feps(x_\delta))^2\mu_1(\delta)\sec^2(\theta)+\frac{\mu_1(\delta)}{\sinh^ 2(\eta/2)},
$$
and
\begin{multline*}
A_8=\text{P.V.}\left(\int_{B(0,1)}+\int_{B^c(0,1)}\right)\frac{\left(\pax\feps(x_\delta)\mu_1^2(\delta)+\mu_1(\delta)(1-(\pax\feps(x_\delta))^2)-\pax\feps(x_\delta))\right)}{\cos^2(\theta)\sinh^2(\eta/2)\left(1+\frac{\tan^2(\theta)|\tanh\left(\eta/2\right)|^{2\epsilon}}{\tanh^2\left(\eta/2\right)}\right)^2}d\eta\\
=B_{11}+B_{12}.
\end{multline*}
Therefore, as in \cite{CGO}, the sign of $A_8$ is the same as the sign of 
$$
Q_1(\mu_1(\delta))=\pax\feps(x_\delta)\mu_1^2(\delta)+\mu_1(\delta)(1-\left(\pax\feps(x_\delta))^2\right)-\pax\feps(x_\delta).
$$
The roots of $Q_1$ are $\pax f(x_\delta)$ and $-1/\pax f(x_\delta)$, so, if we have
$$
\left|\mu_1(\delta)\right|\leq \min\left\{\|\pax \feps(\delta)\|_{L^\infty},\frac{1}{\|\pax \feps(\delta)\|_{L^\infty}}\right\},
$$
then we can ensure that this contribution is negative. Since \eqref{IVpropmu1delta}, we get
$$
B_{12}=\text{P.V.}\int_{B^c(0,1)}\frac{\sec^2(\theta)\left(\pax\feps(x_\delta)\mu_1^2(\delta)+\mu_1(\delta)(1-(\pax\feps(x_\delta))^2)-\pax\feps(x_\delta))\right)}{\sinh^2(\eta/2)\left(1+\frac{\tan^2(\theta)|\tanh\left(\eta/2\right)|^{2\epsilon}}{\tanh^2\left(\eta/2\right)}\right)^2}d\eta<0.
$$
Using the cancellation when $\mu_1(\delta)=\pax f(x_\delta)$, we obtain
\begin{equation}\label{IVeqi1}
B_{11}=\text{P.V.}\int_{B(0,1)}\frac{Q_1(\mu_1(\delta))}{\cos^ 2(\theta)\sinh^2(\eta/2)\left(1+\frac{\tan^2(\theta)|\tanh\left(\eta/2\right)|^{2\epsilon}}{\tanh^2\left(\eta/2\right)}\right)^2}d\eta,
\end{equation}
where
$$
Q_1(\mu_1(\delta))=\pax \feps(x_\delta)(\mu_1^2(\delta)-(\pax \feps(x_\delta))^2)+(1-(\pax \feps(x_\delta))^2)(\mu_1(\delta)-\pax \feps(x_\delta)).
$$
We remark that $\mu_1(\delta)-\pax \feps(x_\delta)<\mu_1(\delta)+\pax \feps(x_\delta)$. We consider the cases given by the sign and the size of $\mu_1(\delta)$.

\emph{1. Case $\mu_1(\delta)>\pax f(x_\delta)$:} In this case, we have $\mu_1(\delta)-\pax f(x_\delta)>0$ and $\mu_1(\delta)+\pax f(x_\delta)>0$. Using the definition of $\theta$ in \eqref{IVtheta} and the fact that $|\eta|\leq 1$, we have \eqref{IVfact3} (see Lemma \ref{IVlemmu}). Notice that, in this case, we have $\mu_1^2(\delta)-(\pax f(x_\delta))^2>0$ and we get \eqref{IVfact4}. Due to \eqref{IVfact3} and \eqref{IVfact4} we obtain
\begin{multline}\label{IVeqI1bound}
B_{11}\leq\frac{\left(\pax \feps(x_\delta)+5(\pax \feps(x_\delta))^3\right)\left(1+\pax \feps(x_\delta)\left(\pax \feps(x_\delta)+\frac{\tan\left(\frac{\pax \feps(x_\delta)}{2}\right)}{\tanh\left(\frac{1}{2}\right)}\right)\right)}{48\tanh(1/2)\cos^2(\|\feps(\delta)\|_{L^\infty}(\RR))}\\
\times\int_{B(0,1)}\frac{\eta^2d\eta}{\sinh^2\left(\frac{\eta}{2}\right)}\\
\leq \frac{\left(\pax \feps(x_\delta)+5(\pax \feps(x_\delta))^3\right)\left(1+\pax \feps(x_\delta)\left(\pax \feps(x_\delta)+\frac{\tan\left(\frac{\pax \feps(x_\delta)}{2}\right)}{\tanh\left(\frac{1}{2}\right)}\right)\right)}{6\tanh(1/2)\cos^2(\|\feps(\delta)\|_{L^\infty(\RR)})}.
\end{multline}

\emph{2. Case $-\pax \feps(x_\delta)<\mu_1(\delta)<\pax \feps(x_\delta)>0$:} In this case we have $\mu_1(\delta)-\pax \feps(x_\delta)\leq0$ and $\mu_1(\delta)+\pax \feps(x_\delta)>0$. Therefore, we get $B_{11}<0$ and we can neglect it. 

\emph{3. Case $\mu_1(\delta)<-\pax \feps(x_\delta)$:} We remark that in this case we have $\mu_1(\delta)-\pax \feps(x_\delta)\leq0$ and $\mu_1(\delta)+\pax \feps(x_\delta)\leq0$. We split 
\begin{equation}\label{IVspliteqcan2}
\mu_1(\delta)+\pax \feps(x_\delta)=\frac{\tan(\theta)-\theta}{\tanh(\eta/2)}
+\theta\left(\frac{1}{\tanh(\eta/2)}-\frac{2}{\eta}\right)+\frac{2\theta}{\eta}+\pax \feps(x_\delta).
\end{equation}
The last term is now positive due to the definition of $\pax \feps(x_\delta)$. Then, in this case, we have
$$
\pax \feps(x_\delta)(\mu_1(\delta)-\pax f(x_\delta))\left(\frac{2\theta}{x_t-\eta}+\pax \feps(x_t)\right)\leq0,
$$
and we can neglect its contribution. Using Taylor's theorem in \eqref{IVspliteqcan2} we obtain the bound \eqref{IVfact4} and \eqref{IVeqI1bound}. 
\end{proof}

We are done with $\mathcal{I}_1$ in \eqref{IVeqreg35} and now we move on to $\mathcal{I}_2$. These terms are easier because the integrals are not singular. With the same ideas as before we can bound the term involving $\Omega_2$:
\begin{lem}\label{IVlemdf7}
The contribution of $\Omega_2$ is bounded by
\begin{equation*}
\epsilon\left|\int_\RR\Omega_2d\eta\right|\leq \epsilon c\sec^4\left(\|f_0\|_{L^\infty(\RR)}\right)\pax\feps(x_\delta).
\end{equation*}
\end{lem}
\begin{proof}
The proof is straightforward.
\end{proof}

We are left with $\Omega_1$ in \eqref{IVeqOmega1}. First, we consider
$$
A_{9}=\int_\RR\frac{\Omega_1^1+\Omega_1^2}{\left(1+\frac{\tan^2(\bar{\theta})\tanh^2\left((x_t-\eta)/2\right)}{|\tanh\left((x_t-\eta)/2\right)|^{2\epsilon}}\right)^2}d\eta.
$$
\begin{lem}\label{IVlemdf8}
The term $A_9$ is bounded as
\begin{equation*}
|A_9|\leq 4\sec^ 2\left(\|f_0\|_{L^\infty(\RR)}\right)\left(\tan\left(\|\feps(\delta)\|_{L^\infty(\RR)}\right)+\pax\feps(x_\delta)\right)
\end{equation*}
\end{lem}
\begin{proof}
Using classical trigonometric identities, we compute
\begin{multline}\label{IVeqI2bound}
A_{9}=\int_\RR \frac{-\pax\feps(x_\delta)\mu_2^2(\delta)+\left((\pax\feps(x_\delta))^2-1\right)\mu_2(\delta)
+\pax\feps(x_\delta)}{\cosh^2(\eta/2)\cos^2(\bar{\theta})\left(1+\frac{\tan^2(\bar{\theta})\tanh^2\left((x_t-\eta)/2\right)}{|\tanh\left((x_t-\eta)/2\right)|^{2\epsilon}}\right)^2}d\eta\\
\leq 4\sec^ 2\left(\|f_0\|_{L^\infty(\RR)}\right)\left(\tan\left(\|\feps(\delta)\|_{L^\infty(\RR)}\right)+\pax\feps(x_\delta)\right).
\end{multline}
\end{proof}

We have to bound the terms containing $\Omega_1^i$. These terms are 
$$
A_{10}=\int_\RR\frac{\left(|\tanh(\eta/2)|^{-3\epsilon}-1\right)\Omega_1^1}{\left(1+\frac{\tan^2(\bar{\theta})\tanh^2\left(\eta/2\right)}{|\tanh\left(\eta/2\right)|^{2\epsilon}}\right)^2}d\eta \text{ and }A_{11}=\int_\RR\frac{\left(|\tanh(\eta/2)|^{-\epsilon}-1\right)\Omega_1^2}{\left(1+\frac{\tan^2(\bar{\theta})\tanh^2\left(\eta/2\right)}{|\tanh\left(\eta/2\right)|^{2\epsilon}}\right)^2}d\eta.
$$
To obtain the decay with $\epsilon$ we split the integral in the regions $B(0,\epsilon)$ and $B^c(0,\epsilon)$ as before. 
\begin{lem}\label{IVlemdf9}
The terms $A_{10}$ and $A_{11}$ are bounded by
\begin{equation*}
|A_{10}|+|A_{11}| \leq c\pax\feps(x_\delta)\sec^ 4\left(\|f_0\|_{L^\infty(\RR)}\right)\left(1+\tan\left(\|f_0\|_{L^\infty(\RR)}\right)\right)\left(\epsilon^{7/10}+\epsilon\right)
\end{equation*}
\end{lem}
\begin{proof}
Using this splitting, $0<\epsilon<1/10$, \eqref{IVfact1}, \eqref{IVfact1.b} and \eqref{IVH4}, we get
\begin{equation*}
A_{10}\leq c\pax\feps(x_\delta)\sec^ 4\left(\|f_0\|_{L^\infty(\RR)}\right)\left(1+\tan\left(\|f_0\|_{L^\infty(\RR)}\right)\right)\left(\int_0^\epsilon\frac{d\eta}{|\tanh(\eta/2)|^{3/10}}+\epsilon\right).
\end{equation*}
With the same ideas and using \eqref{IVH3}, we have 
\begin{equation*}
A_{11}\leq c\pax\feps(x_\delta)\sec^ 2\left(\|f_0\|_{L^\infty(\RR)}\right)\left(1+\tan\left(\|f_0\|_{L^\infty(\RR)}\right)\right)\left(\int_0^\epsilon\frac{d\eta}{|\tanh(\eta/2)|^{1/10}}+\epsilon\right).
\end{equation*}
In order to estimate the decay with $\epsilon$ of these integrals we compute
$$
\int_0^\epsilon\frac{1}{|\tanh(\eta/2)|^{3/10}}-\frac{1}{|\eta/2|^{3/10}}d\eta+\int_0^\epsilon\frac{d\eta}{|\eta/2|^{3/10}}\leq \epsilon+2\epsilon^{7/10},
$$
and
$$
\int_0^\epsilon\frac{1}{|\tanh(\eta/2)|^{/10}}-\frac{1}{|\eta/2|^{1/10}}d\eta+\int_0^\epsilon\frac{d\eta}{|\eta/2|^{1/10}}\leq \epsilon+2\epsilon^{9/10}.
$$
\end{proof}

We have the following result concerning the evolution of the slope:
\begin{prop}\label{IVMPdf}
Let $f_0\in  W^{1,\infty}(\RR)$ be the initial datum in \eqref{IVdefi} satisfying \eqref{IVH3}, \eqref{IVH4} and \eqref{IVH5}, define $\feps_0$ as in \eqref{IVinitial} and let $\feps$ be the classical solution of \eqref{IVeqreg} corresponding to the initial datum $\feps_0$. Then $\feps$ verifies
$$
\|\pax \feps(t)\|_{L^\infty(\RR)}\leq \|\pax \feps_0\|_{L^\infty(\RR)}\leq \|\pax f_0\|_{L^\infty(\RR)}<1.
$$
\end{prop}

\begin{proof}
For the sake of simplicity we split the proof in different steps. 

\textbf{Step 1 (local decay):}  Combining $B_{11}$ in \eqref{IVeqi1} and $A_9$ in Lemma \ref{IVlemdf8}, and using the bounds \eqref{IVeqI1bound} and \eqref{IVeqI2bound} and the hypothesis \eqref{IVH5} we obtain 
$$
B_{11}+|A_9|<0.
$$
We take $\alpha_4=2\sec^2\left(\|f_0\|_{L^\infty(\RR)}\right)$, $\alpha_3=2$, $\alpha_2=3\left(1+\sec^2\left(\|f_0\|_{L^\infty(\RR)}\right)\right)$. Since we have a term $\sqrt{\epsilon}$ and $0<\epsilon<1/10$, we can compare the bounds in Lemmas \ref{IVlemdf1}- \ref{IVlemdf9} with $-\sqrt{\epsilon}\alpha_1\pax\feps(x_\delta)$ if $\alpha_1=\alpha_1\left(\|f_0\|_{L^\infty(\RR)}\right)$ is chosen big enough. The universal constant  $c$ in all these bounds can be $c=1000$. We have shown that for every $0<\delta<<1$ small enough, there is local in time decay. As $\delta$ is positive and arbitrary, we have 
$$
\|\pax \feps(t)\|_{L^\infty}\leq\|\pax \feps(0)\|_{L^\infty}, \text{ for }0\leq t<t^*.
$$

\textbf{Step 2 (from local decay to an uniform bound):} 
Then, in the worst case, we have 
$$
\|\pax \feps (t^*)\|_{L^\infty(\RR)}=\|\pax \feps_0\|_{L^\infty(\RR)}\text{ and }\|\feps (t^*)\|_{L^\infty(\RR)}\leq\|\feps_0\|_{L^\infty(\RR)}.
$$
These inequalities ensure that the hypotheses \eqref{IVH3}, \eqref{IVH4} and \eqref{IVH5} hold at time $t=t^*$ and $\|\pax \feps (t)\|_{L^\infty(\RR)}$ decays again.

\textbf{Step 3 (the case where $\feps(x_t)=\min_x \pax \feps(x,t)$):} This case follows the same ideas, and we conclude, thus, the result.
\end{proof}

\begin{prop}\label{IVUBdf}
Let $f_0\in  W^{1,\infty}(\RR)$ be the initial datum in \eqref{IVdefi} satisfying \eqref{IVsisder2} and define $\feps_0$ as in \eqref{IVinitial}. Let $\feps$ be the classical solution of \eqref{IVeqreg} corresponding to the initial datum $\feps_0$. Then, $\feps$ verifies
$$
\|\feps(t)\|_{L^\infty(\RR)}<\|f_0\|_{L^\infty(\RR)},\text{ }\|\pax \feps(t)\|_{L^\infty(\RR)}<1\quad \forall t>0.
$$
\end{prop}
\begin{proof}
The region delimited by $(x(l),y(l))$ is below the region with maximum principle (see \cite{CGO}). Then, in the worst case, at some $t^*>0$ we have that\newline $(\|\feps(t)\|_{L^\infty(\RR)},\|\pax \feps(t)\|_{L^\infty(\RR)})$ fulfills the hypotheses \eqref{IVH3},   \eqref{IVH4} and \eqref{IVH5}. From them the result follows.
\end{proof}

\section{Global existence for $\feps$}\label{IVsecglobal}
In this section we obtain \emph{'a priori'} estimates in $H^3(\RR)$ that ensure the global existence for the regularized systems \eqref{IVeqreg} for initial data satisfying hypotheses \eqref{IVH3}, \eqref{IVH4} and \eqref{IVH5} or \eqref{IVsisder2}. First, notice that if the initial datum satisfies hypotheses \eqref{IVH3}, \eqref{IVH4} and \eqref{IVH5}, by Propositions \ref{IVMPf} and \ref{IVMPdf}, the solution satisfies
\begin{equation}\label{IVbounds}
\|\feps (t)\|_{L^\infty}\leq\|f_0\|_{L^\infty(\RR)} \text{ and }\|\pax \feps (t)\|_{L^\infty(\RR)}\leq 1.
\end{equation}
If the initial datum satisfies \eqref{IVsisder2}, by Propositions \ref{IVMPf} and \ref{IVUBdf}, the solution to the regularized system again satisfies the bounds \eqref{IVbounds}. Then we have the following proposition:
\begin{prop}\label{existence}
Let $f_0\in  W^{1,\infty}(\RR)$ be the initial datum in \eqref{IVdefi} satisfying \eqref{IVH3}, \eqref{IVH4} and \eqref{IVH5} or \eqref{IVsisder2} and define $\feps_0$ as in \eqref{IVinitial}. Then for every $\epsilon>0$ and $T>0$ there exists a solution $\feps(x,t)\in C([0,T],H^3(\RR))$.
\end{prop}
\begin{proof}
We have to bound the $L^2$ norm of the function and its third derivative. We split the proof in different steps.

\textbf{Step 0 (local existence):} 
The local existence follows by classical energy methods as in \cite{c-g07, CGO, bertozzi-Majda}. 

\textbf{Step 1 (the function):} 
We have
$$
\frac{1}{2}\frac{d}{dt}\|\feps (t)\|^2_{L^2(\RR)}=-\sqrt{\epsilon} \alpha_1\|\feps(t)\|^2_{L^2(\RR)}-\epsilon\|\pax\feps(t)\|^2_{L^2(\RR)}- I_1+I_2+I_3
$$
Using \eqref{IVlambdaeps} we get
\begin{multline*}
I_1=\alpha_2\epsilon\int_\RR \feps(x)\Lambda_l^{1-\epsilon}\feps(x)dx+ \alpha_3\epsilon\int_\RR \feps(x)\Lambda_{l}^{1-3\epsilon}\feps(x)dx\\
+\int_\RR \feps(x)\left(\Lambda_l-\Lambda_l^{1-\epsilon}\right)\feps(x)dx+\alpha_4\int_\RR \feps(x)\left(\Lambda_l-\Lambda_l^{1-3\epsilon}\right)\feps(x)dx\geq0
\end{multline*}
and we obtain that the contribution of the linear terms is negative. The nonlinear term $\Xi_1^\epsilon$ defined in \eqref{IVxi1eps} is
\begin{multline*}
I_2=\int_\RR \text{P.V.}\int_\RR \frac{\feps(x) \pax \feps(x)\sec^2(\theta)\frac{|\tanh(\eta/2)|^\epsilon}{\tanh(\eta/2)}}{1+\mu^2_1(t)|\tanh(\eta/2)|^{2\epsilon}}d\eta dx\\
+\int_\RR \text{P.V.}\int_\RR \frac{-\feps(x)\pax \feps (x-\eta)\sec^2(\theta)\frac{|\tanh(\eta/2)|^\epsilon}{\tanh(\eta/2)}}{1+\mu^2_1(t)|\tanh(\eta/2)|^{2\epsilon}}d\eta dx=A_1+A_2.
\end{multline*}
Using the cancellation coming from the principal value we have
\begin{multline*}
A_1=\int_\RR\feps(x) \pax \feps(x) \text{P.V.}\int_\RR \frac{|\tanh(\eta/2)|^\epsilon}{\tanh(\eta/2)}\left(\frac{\sec^2(\theta)}{1+\mu^2_1(t)|\tanh(\eta/2)|^{2\epsilon}}-1\right)d\eta dx\\
=\int_\RR\feps(x) \pax \feps(x) \text{P.V.}\int_\RR \frac{|\tanh(\eta/2)|^\epsilon}{\tanh(\eta/2)}\frac{\frac{-\tan^2(\theta)}{\sinh^2(\eta/2)}}{1+\mu^2_1(t)|\tanh(\eta/2)|^{2\epsilon}}d\eta dx\\
\qquad\qquad\qquad+\int_\RR\feps(x) \pax \feps(x) \text{P.V.}\int_\RR \frac{|\tanh(\eta/2)|^\epsilon}{\tanh(\eta/2)}\frac{\mu^2_1(t)(1-|\tanh(\eta/2)|^{2\epsilon})}{1+\mu^2_1(t)|\tanh(\eta/2)|^{2\epsilon}}d\eta dx
\end{multline*}
Inserting \eqref{IVmu1} and \eqref{IVfact2} in the expression for $A_1$ we obtain
$$
|A_1|\leq c(\epsilon)\|\feps(t)\|_{L^2(\RR)}\|\pax \feps(t)\|_{L^2(\RR)}\left(\tan\left(\|f_0\|_{L^\infty(\RR)}\right)+1\right)^4.
$$
The second term in $I_2$ is
\begin{multline*}
A_2=-\int_\RR\text{P.V.}\int_\RR \feps(x) \pax \feps(x-\eta) \frac{|\tanh(\eta/2)|^\epsilon}{\tanh(\eta/2)}\left(\frac{\sec^2(\theta)}{1+\mu^2_1(t)|\tanh(\eta/2)|^{2\epsilon}}-1+1\right)d\eta dx\\
=-\int_\RR \text{P.V.}\int_\RR \feps(x) \pax \feps(x-\eta)\frac{|\tanh(\eta/2)|^\epsilon}{\tanh(\eta/2)}\frac{\frac{-\tan^2(\theta)}{\sinh^2(\eta/2)}}{1+\mu^2_1(t)|\tanh(\eta/2)|^{2\epsilon}}d\eta dx\\
\qquad\qquad\qquad+\int_\RR\text{P.V.}\int_\RR \feps(x) \pax \feps(x-\eta)\frac{|\tanh(\eta/2)|^\epsilon}{\tanh(\eta/2)}\frac{\mu^2_1(t)(1-|\tanh(\eta/2)|^{2\epsilon})}{1+\mu^2_1(t)|\tanh(\eta/2)|^{2\epsilon}}d\eta dx\\
+\int_\RR \text{P.V.}\int_\RR \feps(x) \pax \feps(x-\eta)\frac{|\tanh(\eta/2)|^\epsilon}{\tanh(\eta/2)}d\eta dx.
\end{multline*}
Using the Cauchy–Schwarz inequality, the equality $\pax \feps(x-\eta)=-\partial_\eta\feps(x-\eta)$ and integrating by parts we get
$$
|A_2|\leq c(\epsilon)\|\feps(t)\|_{L^2(\RR)}\|\pax \feps(t)\|_{L^2(\RR)}\left(\tan\left(\|f_0\|_{L^\infty(\RR)}\right)+1\right)^4+c(\epsilon)\|\feps(t)\|^2_{L^2(\RR)}.
$$
To finish with the $L^2$ norm we have to deal with $I_3$. We have
\begin{multline*}
I_3=\int_\RR \text{P.V.}\int_\RR \frac{\feps(x) \pax \feps(x)\sec^2(\bar{\theta})\frac{\tanh(\eta/2)}{|\tanh(\eta/2)|^\epsilon}}{1+\frac{\mu^2_2(t)}{|\tanh(\eta/2)|^{2\epsilon}}}d\eta dx\\
+\int_\RR \text{P.V.}\int_\RR \frac{\feps(x)\pax \feps (x-\eta)\sec^2(\bar{\theta})\frac{\tanh(\eta/2)}{|\tanh(\eta/2)|^\epsilon}}{1+\frac{\mu^2_2(t)}{|\tanh(\eta/2)|^{2\epsilon}}}d\eta dx,
\end{multline*}
where $\bar{\theta}$ is defined in \eqref{IVtheta}.
Using the same ideas as in $I_2$ and
$$
|\mu_2(t)|\leq \tan\left(\|f_0\|_{L^\infty}\right),
$$
we conclude the bound
$$
|I_3|\leq c(\epsilon)\|\feps(t)\|_{L^2(\RR)}\|\pax \feps(t)\|_{L^2(\RR)}\left(\tan\left(\|f_0\|_{L^\infty(\RR)}\right)+1\right)^4+c(\epsilon)\|\feps(t)\|^2_{L^2(\RR)}.
$$
Putting all these bounds together we get
\begin{equation}\label{IVL2}
\frac{d}{dt}\|\feps (t)\|^ 2_{L^2(\RR)}\leq c(\epsilon)\|\feps(t)\|_{L^2(\RR)}\|\pax \feps(t)\|_{L^2(\RR)}\left(\tan\left(\|f_0\|_{L^\infty(\RR)}\right)+1\right)^4+c(\epsilon)\|\feps(t)\|^2_{L^2(\RR)}.
\end{equation}

\textbf{Step 2 (the third derivative):} 
To study the $L^2$ norm of the third derivative, we compute
$$
\frac{1}{2}\frac{d}{dt}\|\pax^3\feps(t)\|^2_{L^2(\RR)}=-\sqrt{\epsilon} \alpha_1\|\pax^ 3 \feps(t)\|^ 2_{L^2(\RR)}-\epsilon\|\pax^ 4\feps(t)\|_{L^2(\RR)}^2- I_4+I_5+I_6.
$$
The term $I_4$ is positive due to Lemma \ref{IVops}:
\begin{multline*}
I_4=\alpha_2\epsilon\int_\RR \pax^3\feps(x)\Lambda_l^{1-\epsilon}\pax^3\feps(x)dx+ \alpha_3\epsilon\int_\RR \pax^3\feps(x)\Lambda_{l}^{1-3\epsilon}\pax^ 3\feps(x)dx\\
\int_\RR \pax^3\feps(x)\left(\Lambda_l-\Lambda_l^{1-\epsilon}\right)\pax^3\feps(x)dx+\alpha_4\int_\RR \pax^3\feps(x)\left(\Lambda_l-\Lambda_l^{1-3\epsilon}\right)\pax^3\feps(x)dx\geq0.
\end{multline*}
The nonlinear terms related to $\theta$ are
$$
I_5=-\int_\RR \pax^4\feps(x)\pax^2\left[ \text{P.V.}\left(\int_{B(0,1)}+\int_{B^c(0,1)}\right) \frac{2\pax\theta\sec^2(\theta)\frac{|\tanh(\eta/2)|^\epsilon}{\tanh(\eta/2)}}{1+\mu^2_1(t)|\tanh(\eta/2)|^{2\epsilon}}d\eta\right] dx=A_3+A_4.
$$
The term $A_3$ is not singular if $\epsilon>0$ and can be bounded using H\"{o}lder and Nirenberg interpolation inequalities. For the sake of brevity, we write some terms detailedly, being the rest analogous to them. We have
$$
A_3=B_1+B_2+\text{ lower order terms }.
$$
Using 
$$
\feps(x)-\feps(x-\eta)=\eta\int_0^1\pax^2\feps(x+(s-1)\eta)ds,
$$
we obtain
\begin{eqnarray*}
B_1&=&\int_\RR \pax^4\feps(x) \text{P.V.}\int_{B(0,1)} \frac{16(\pax\theta)^3\sec^6(\theta)\frac{|\tanh(\eta/2)|^{5\epsilon}}{\tanh^5(\eta/2)}\tan^2(\theta)}{\left(1+\mu^2_1(t)|\tanh(\eta/2)|^{2\epsilon}\right)^3}d\eta dx\\
&=&\int_0^1\int_0^1\int_\RR \pax^4\feps(x) \text{P.V.}\int_{B(0,1)} \frac{4\pax^2\feps(x+(s-1)\eta)\pax^2\feps(x+(r-1)\eta)\pax\theta \eta^2d\eta dx dr ds}{\left(1+\mu^2_1(t)|\tanh(\eta/2)|^{2\epsilon}\right)^3\cos^ 6(\theta)\frac{\tanh^5(\eta/2)}{|\tanh(\eta/2)|^{5\epsilon}}\cot^2(\theta)}\\
&\leq&\|\pax^4\feps(t)\|_{L^2}\|\pax^2\feps(t)\|^2_{L^4(\RR)}c\sec^6\left(\|f_0\|_{L^\infty(\RR)}\right)\int_{B(0,1)}|\tanh(\eta/2)|^{5\epsilon-5}\eta^2\tan^2(\eta/2)d\eta.
\end{eqnarray*}
The second term is 
\begin{eqnarray*}
B_2&=&-\int_\RR \pax^4\feps(x) \text{P.V.}\int_{B(0,1)} \frac{4(\pax\theta)^3\sec^6(\theta)\frac{|\tanh(\eta/2)|^{3\epsilon}}{\tanh^3(\eta/2)}}{\left(1+\mu^2_1(t)|\tanh(\eta/2)|^{2\epsilon}\right)^2}d\eta dx\\
&=&\int_0^1\int_0^1\int_\RR \pax^4\feps(x) \text{P.V.}\int_{B(0,1)} \frac{\pax^2\feps(x+(s-1)\eta)\pax^2\feps(x+(r-1)\eta)\pax\theta \eta^2d\eta dx dr ds}{\left(1+\mu^2_1(t)|\tanh(\eta/2)|^{2\epsilon}\right)^3\cos^ 6(\theta)\frac{\tanh^5(\eta/2)}{|\tanh(\eta/2)|^{5\epsilon}}}\\
&\leq&\|\pax^4\feps(t)\|_{L^2}\|\pax^2\feps(t)\|^2_{L^4(\RR)}c\sec^6\left(\|f_0\|_{L^\infty(\RR)}\right)\int_{B(0,1)}|\tanh(\eta/2)|^{3\epsilon-3}\eta^2 d\eta,
\end{eqnarray*}
and using the classical interpolation inequality
$$
\|\pax^2 f\|_{L^4(\RR)}^2\leq c\|\pax f\|_{L^\infty(\RR)}\|\pax^3 f\|_{L^2(\RR)},
$$
we get
$$
|A_3|\leq c(\epsilon)\|\pax^4\feps (t)\|_{L^2(\RR)}\|\feps(t)\|_{H^3(\RR)}\left(1+\sec\left(\|f_0\|_{L^\infty(\RR)}\right)\right)^6.
$$
We split the term $A_4$ as follows
\begin{multline*}
A_4=\int_\RR \pax^4\feps(x)\pax^2 \text{P.V.}\int_{B^c(0,1)} \frac{2\pax\theta\sec^2(\theta)}{\tanh(\eta/2)}\left(\frac{1}{1+\mu^2_1(t)|\tanh(\eta/2)|^{2\epsilon}}-\frac{1}{1+\mu^2_1(t)}\right)d\eta dx\\
+\int_\RR \pax^4\feps(x)\pax^2 \text{P.V.}\int_{B^c(0,1)} \frac{2\pax\theta\sec^2(\theta)\frac{|\tanh(\eta/2)|^\epsilon-1}{\tanh(\eta/2)}}{1+\mu^2_1(t)|\tanh(\eta/2)|^{2\epsilon}}d\eta dx\\
+\int_\RR \pax^4\feps(x)\pax^2 \text{P.V.}\int_{B^c(0,1)} \frac{2\pax\theta\sec^2(\theta)}{\tanh(\eta/2)}\frac{1}{1+\mu^2_1(t)}d\eta dx=B_3+B_4+B_5.
\end{multline*}
These terms are not singular because of the domain of integration. We have to deal with the integrability at infinity in $\eta$. We compute
$$
B_3=\int_\RR \pax^4\feps(x)\pax^2 \text{P.V.}\int_{B^c(0,1)} \frac{2\pax\theta\sec^2(\theta)}{\tanh(\eta/2)}\frac{\mu_1^2(t)\left(1-|\tanh(\eta/2)|^{2\epsilon}\right)}{\left(1+\mu^2_1(t)|\tanh(\eta/2)|^{2\epsilon}\right)\left(1+\mu^2_1(t)\right)}d\eta dx.
$$
The integrability at infinity is obtained using \eqref{IVfact1} and \eqref{IVfact1.b}. We only bound the more singular terms in $B_3$ and $B_4$. The most singular term in $B_3$ is
$$
C_1=\int_\RR \pax^4\feps(x)\text{P.V.}\int_{B^c(0,1)} \frac{2\pax^3\theta\sec^2(\theta)}{\tanh(\eta/2)}\frac{\mu_1^2(t)\left(1-|\tanh(\eta/2)|^{2\epsilon}\right)}{\left(1+\mu^2_1(t)|\tanh(\eta/2)|^{2\epsilon}\right)\left(1+\mu^2_1(t)\right)}d\eta dx.
$$
Using \eqref{IVfact1}, \eqref{IVfact1.b} and \eqref{IVfact2}, we obtain
$$
|C_1|\leq c\|\pax^4\feps (t)\|_{L^2(\RR)}\|\feps(t)\|_{H^3(\RR)}\tan\left(\|f_0\|_{L^\infty}\right)\sec^2\left(\|f_0\|_{L^\infty}\right).
$$
Analogously, the more singular term in $B_4$ is
$$
C_2=\int_\RR \pax^4\feps(x)\text{P.V.}\int_{B^c(0,1)} \frac{2\pax^3\theta\sec^2(\theta)\frac{|\tanh(\eta/2)|^\epsilon-1}{\tanh(\eta/2)}}{1+\mu^2_1(t)|\tanh(\eta/2)|^{2\epsilon}}d\eta dx.
$$
Using the same bounds as in $C_1$, we get
$$
|C_2|\leq c\|\pax^4\feps (t)\|_{L^2(\RR)}\|\feps(t)\|_{H^3(\RR)}\sec^2\left(\|f_0\|_{L^\infty}\right).
$$
Using classical trigonometric identities, we obtain
$$
B_5=\int_\RR \pax^4\feps(x)\pax^2 \text{P.V.}\int_{B^c(0,1)} \frac{2\pax\theta\sinh(\eta)}{\cosh(\eta)-\cos(2\theta)}d\eta.
$$
And the most singular term in $B_5$ is
\begin{multline*}
C_3=\int_\RR \pax^4\feps(x)\text{P.V.}\int_{B^c(0,1)} \frac{\pax^ 3 \feps(x)\sinh(\eta)}{2\sinh^2(\eta/2)\left(1+\frac{\sin^2(\theta)}{\sinh^2(\eta/2)}\right)}d\eta\\
-\int_\RR \pax^4\feps(x)\text{P.V.}\int_{B^c(0,1)} \frac{\pax^ 3 \feps(x-\eta)\sinh(\eta)}{\cosh(\eta)-\cos(2\theta)}d\eta=D_1+D_2.
\end{multline*}
Using the cancellation of the principal value integral we obtain
$$
D_1=-\int_\RR \pax^4\feps(x)\pax^ 3 \feps(x)\text{P.V.}\int_{B^c(0,1)} \frac{\sin^2(\theta)\sinh(\eta)}{2\sinh^4(\eta/2)\left(1+\frac{\sin^2(\theta)}{\sinh^2(\eta/2)}\right)}d\eta,
$$
thus,
$$
|D_1|\leq c\|\pax^4\feps (t)\|_{L^2(\RR)}\|\feps(t)\|_{H^3(\RR)}.
$$
Integrating by parts in $D_2$, we obtain the required decay at infinity and we conclude
$$
|D_2|\leq c\|\pax^4\feps (t)\|_{L^2(\RR)}\|\feps(t)\|_{H^3(\RR)}.
$$
Putting all together, we get
$$
|I_5|\leq c(\epsilon)\|\pax^4\feps (t)\|_{L^2(\RR)}\|\feps(t)\|_{H^3(\RR)}\left(1+\sec\left(\|f_0\|_{L^\infty}\right)\right)^6.
$$
The nonlinear terms related to $\bar{\theta}$ are
$$
I_6=-\int_\RR \pax^4\feps(x)\pax^2 \text{P.V.}\left(\int_{B(0,1)}+\int_{B^c(0,1)}\right) \frac{2\pax\bar{\theta}\sec^2(\bar{\theta})\frac{\tanh(\eta/2)}{|\tanh(\eta/2)|^\epsilon}}{1+\mu^2_2(t)|\tanh(\eta/2)|^{-2\epsilon}}d\eta dx=A_5+A_6.
$$
We observe that, due to $1/10>\epsilon>0$ and $\|f_0\|_{L^\infty(\RR)}<\pi/2$, this integral is not singular. Thus the inner part $A_5$ can be bounded following the same ideas as for $A_3$. The integrability at infinity is obtained with the following splitting
\begin{multline*}
A_6=-\int_\RR \pax^4\feps(x)\pax^2 \text{P.V.}\int_{B^c(0,1)}\left(\frac{2\pax\bar{\theta}\sec^2(\bar{\theta})\tanh(\eta/2)}{1+\mu^2_2(t)|\tanh(\eta/2)|^{-2\epsilon}}-\frac{2\pax\bar{\theta}\sec^2(\bar{\theta})\tanh(\eta/2)}{1+\mu^2_2(t)}\right)d\eta dx\\
-\int_\RR \pax^4\feps(x)\pax^2 \text{P.V.}\int_{B^c(0,1)}\frac{2\pax\bar{\theta}\sec^2(\bar{\theta})\tanh(\eta/2)\left(\frac{1}{|\tanh(\eta/2)|^\epsilon}-1\right)}{1+\mu^2_2(t)|\tanh(\eta/2)|^{-2\epsilon}}d\eta dx\\
-\int_\RR \pax^4\feps(x)\pax^2 \text{P.V.}\int_{B^c(0,1)}\frac{2\pax\bar{\theta}\sec^2(\bar{\theta})\tanh(\eta/2)}{1+\mu^2_2(t)}d\eta dx=B_6+B_7+B_8.
\end{multline*}
The term $B_8$ is
\begin{multline*}
B_8=-\int_\RR \pax^4\feps(x)\pax^2 \text{P.V.}\int_{B^c(0,1)}\frac{\pax\bar{\theta}\sinh(\eta)}{\cosh(\eta)+\cos(2\bar{\theta})}d\eta dx\\
=-\int_\RR \pax^4\feps(x)\pax^2 \text{P.V.}\int_{B^c(0,1)}\frac{\pax\bar{\theta}\sinh(\eta)}{2\sinh^2(\eta/2)\left(1+\frac{\cos^2(\bar{\theta})}{\sinh^2(\eta/2)}\right)}d\eta dx,
\end{multline*}
and it can be handled as $B_5$. The terms $B_6$ and $B_7$ have a term $|\tanh(\eta/2)|^{k\epsilon}-1|$ and they can be bounded following the steps in $B_3$ and $B_4$ by using \eqref{IVfact1} and \eqref{IVfact1.b}. Putting all the estimates together we obtain
$$
|I_6|\leq c(\epsilon)\|\pax^4\feps (t)\|_{L^2(\RR)}\|\feps(t)\|_{H^3(\RR)}\left(1+\sec\left(\|f_0\|_{L^\infty}\right)\right)^6.
$$
Using \eqref{IVL2}, Young's inequality and the dissipation given by the Laplacian we get the \emph{'a priori'} estimate
\begin{equation}\label{IVnormH3}
\frac{d}{dt}\|\feps (t)\|_{H^3(\RR)}^2\leq c(\epsilon)\|\feps (t)\|_{H^3(\RR)}^2C\left(\|f_0\|_{L^\infty(\RR)},\|\pax f_0\|_{L^\infty(\RR)}\right).
\end{equation}
A classical continuation argument shows the global existence. 
\end{proof}

\section{Convergence of $\feps$}\label{IVsec4}
In this section we study the limit of $\feps$ as $\epsilon\rightarrow0$. 

\begin{lem}\label{IVlemafeps}
The regularized solutions $\feps$ corresponding to an initial datum satisfying the hypotheses \eqref{IVH3}, \eqref{IVH4} and \eqref{IVH5}, or \eqref{IVsisder2}, converge (up to a subsequence) weakly-* to $f\in L^\infty([0,T],W^{1,\infty}(\RR))$. Moreover, up to a subsequence, $\feps\rightarrow f$ in $L^\infty(K)$ for all compact set $K\subset\RR\times\RR^+$.
\end{lem}
\begin{proof}
First, notice that, due to Propositions \ref{IVMPf}, \ref{IVMPdf} and \ref{IVUBdf} and hypotheses \eqref{IVH3}, \eqref{IVH4} and \eqref{IVH5}, the regularized solutions satisfy 
$$
\|\feps (t)\|_{L^\infty(\RR)}\leq \|f_0\|_{L^\infty(\RR)}<\frac{\pi}{4},\quad \|\pax \feps (t)\|_{L^\infty(\RR)}\leq\|\pax f_0\|_{L^\infty(\RR)},
$$
while, if the initial datum, instead of hypotheses \eqref{IVH3}, \eqref{IVH4} and \eqref{IVH5}, satisfies \eqref{IVsisder2} then
$$
\|\feps (t)\|_{L^\infty(\RR)}\leq \|f_0\|_{L^\infty(\RR)},\quad \|\pax \feps (t)\|_{L^\infty(\RR)}\leq1.
$$
Due to the Banach-Alaoglu Theorem, these bounds imply that there exists a subsequence such that
$$
\int_0^T\int_\RR \feps(x,t)g(x,t) dxdt\rightarrow\int_0^T\int_\RR f(x,t)g(x,t) dxdt,
$$
and
$$
\int_0^T\int_\RR \pax\feps(x,t)g(x,t) dxdt\rightarrow\int_0^T\int_\RR \pax f(x,t)g(x,t) dxdt,
$$
with $f\in L^\infty([0,T],W^{1,\infty}(\RR))$, any $g\in L^1([0,T]\times\RR)$ and every $T>0$. Fixing $t$, due to the uniform bound in $W^{1,\infty}(\RR)$ and the Ascoli-Arzela Theorem we have that, up to a subsequence, $\feps(t)\rightarrow f(t)$ uniformly on any bounded interval $I\subset\RR$. Moreover, for all $N$, we have
$$
\|\feps -f\|_{L^\infty(B(0,N)\times[0,T])}\rightarrow0.
$$
In order to prove this uniform convergence on compact sets we use the spaces and results contained in \cite{ccgs-10}. For $v\in L^\infty(B(0,N))$, we define the norm
\begin{equation}\label{norm-2}
\|v\|_{W_*^{-2,\infty}(B(0,N))}=\sup_{\phi\in W^{2,1}_0(B(0,N)),\|\phi\|_{W^{2,1}}\leq 1}\left|\int_{B(0,N)}\phi(x)v(x)dx\right|.
\end{equation}
We define the Banach space $W^{-2,\infty}_*(B(0,N))$ as the completion of $L^\infty(B(0,N))$ with respect to the norm \eqref{norm-2}. We have
$$
W^{1,\infty}(B(0,N))\subset L^\infty(B(0,N))\subset W^{-2,\infty}_*(B(0,N)).
$$
The embedding $L^\infty(B(0,N))\subset W^{-2,\infty}_*(B(0,N))$ is continuous and, due to the Ascoli-Arzela Theorem, the embedding $W^{1,\infty}(B(0,N))\subset L^\infty(B(0,N))$ is compact. We use the following Lemma
\begin{lem}[\cite{ccgs-10}]\label{IVlemapaco}
Consider a sequence $\{u_m\}\in C([0,T]\times B(0,N))$ that is uniformly bounded in the space $L^\infty([0,T],W^{1,\infty}(B(0,N)))$. Assume further that the weak derivative $du_m/dt$ is in $L^\infty([0,T],L^\infty(B(0,N)))$ (not necessarily uniform) and is uniformly bounded in \makebox{$L^\infty([0,T],W^{-2,\infty}_*(B(0,N)))$}. Finally suppose that $\pax u_m\in C([0,T]\times B(0,N))$. Then there exists a subsequence of $u_m$ that converges strongly in $L^\infty([0,T]\times B(0,N)).$
\end{lem}

Due to this Lemma we only need to bound $\pat \feps$ in $L^\infty([0,T]\times B(0,N))$ (not uniformly) and in $L^\infty([0,T],W^{-2,\infty}_*(B(0,N)))$ (uniformly). Using that $\feps\in C([0,T],H^3(\RR))$, the linear terms in \eqref{IVeqreg2} can be bounded easily with a bound depending on $\epsilon$. To bound the nonlinear terms we split the integral
$$
\text{P.V.}\int_\RR=\text{P.V.}\int_{B(0,1)}+\text{P.V.}\int_{B^c(0,1)},
$$
and we compute
\begin{multline*}
\left|\text{P.V.}\int_\RR\frac{\pax\feps(x)\sec^2(\theta)\frac{|\tanh(\eta/2)|^\epsilon}{\tanh(\eta/2)}}{1+\frac{\tan^2(\theta)|\tanh\left(\eta/2\right)|^{2\epsilon}}{\tanh^2\left(\eta/2\right)}}d\eta\right|\leq c(\epsilon)\sec^2(\|f_0\|_{L^\infty(\RR)})\\
+\left|\text{P.V.}\int_{B^c(0,1)}\frac{|\tanh(\eta/2)|^\epsilon}{\tanh(\eta/2)}\frac{-\frac{\tan^2(\theta)}{\sinh^2(\eta/2)}+\mu^2_1(t)(1-|\tanh(\eta/2)|^{2\epsilon})}{1+\frac{\tan^2(\theta)|\tanh\left(\eta/2\right)|^{2\epsilon}}{\tanh^2\left(\eta/2\right)}}d\eta\right|\\
\leq c(\epsilon)\left(\sec^2(\|f_0\|_{L^\infty(\RR)})+\tan^2(\|f_0\|_{L^\infty(\RR)})\right)
\end{multline*}
where we have used $\sec^2(\theta)-1=\tan^2(\theta)$, \eqref{IVfact1.b}, \eqref{IVfact2} and
$$
\text{P.V.}\int_\RR\frac{|\tanh(\eta/2)|^\epsilon}{\tanh(\eta/2)}d\eta=0.
$$
The second term with the kernel involving $\theta$ is
$$
\left|\text{P.V.}\int_\RR\frac{(\epsilon-1)\tan(\theta)\frac{|\tanh((x-\eta)/2)|^\epsilon}{\sinh^2((x-\eta)/2)}}{1+\frac{\tan^2(\theta)|\tanh\left((x-\eta)/2\right)|^{2\epsilon}}{\tanh^2\left((x-\eta)/2\right)}}d\eta\right|\leq c(\epsilon)\left(\tan\left(\|f_0\|_{L^\infty(\RR)}\right)+1\right).
$$
The terms with the kernel involving $\bar{\theta}$ are not singular and can be bounded following the same ideas
\begin{multline*}
\left|\text{P.V.}\int_\RR\frac{\sec^2(\bar{\theta})\frac{\tanh((x-\eta)/2)}{|\tanh((x-\eta)/2)|^\epsilon}}{1+\frac{\tan^2(\bar{\theta})\tanh^2\left((x-\eta)/2\right)}{|\tanh\left((x-\eta)/2\right)|^{2\epsilon}}}d\eta\right|\leq c(\epsilon)\sec^2\left(\|f_0\|_{L^\infty(\RR)}\right)\\
+\left|\text{P.V.}\int_{B^c(0,1)}\frac{\tanh(\eta/2)}{|\tanh(\eta/2)|^\epsilon}\frac{-\frac{\tan^2(\bar{\theta})}{\sinh^2(\eta/2)}+\frac{\mu^2_2(t)}{|\tanh(\eta/2)|^\epsilon}(|\tanh(\eta/2)|^{2\epsilon}-1)}{1+\frac{\tan^2(\bar{\theta})\tanh^2\left((x-\eta)/2\right)}{|\tanh\left((x-\eta)/2\right)|^{2\epsilon}}}d\eta\right|\\
\leq c(\epsilon)\left(\sec^2(\|f_0\|_{L^\infty(\RR)})+\tan^2(\|f_0\|_{L^\infty(\RR)})\right)
\end{multline*}
and
$$
\left|\text{P.V.}\int_\RR\frac{\frac{(1-\epsilon)\tan(\bar{\theta})}{|\tanh((x-\eta)/2)|^\epsilon\cosh^2((x-\eta)/2)}}{1+\frac{\tan^2(\bar{\theta})\tanh^2\left((x-\eta)/2\right)}{|\tanh\left((x-\eta)/2\right)|^{2\epsilon}}}d\eta\right|\leq c(\epsilon)\tan\left(\|f_0\|_{L^\infty(\RR)}\right).
$$
Putting together all these estimates we get
$$
|\pat f^\epsilon(x,t)|\leq c(\epsilon)\left(\|f_0\|_{L^2(\RR)}+\sec^2(\|f_0\|_{L^\infty(\RR)})+\tan^2(\|f_0\|_{L^\infty(\RR)})\right),
$$
thus we conclude with the bound in $L^\infty([0,T]\times B(0,N))$. 

To obtain the bound in $L^\infty([0,T],W^{-2,\infty}_*(B(0,N)))$ we extend $\phi\in W^{2,1}_0(B(0,N))$ by zero outside of this ball of radius $N$. Then, using Lemma \ref{IVops}, we \emph{integrate by parts} and obtain
$$
\int_\RR\phi(x)\Lambda_l^{1-\epsilon}\feps(x)dx
\leq \|\Lambda_l^{1-\epsilon}\phi\|_{L^1(\RR)}\|f_0\|_{L^\infty(\RR)},
$$
$$
\int_\RR\phi(x)\Lambda_l^{1-3\epsilon}\feps(x)dx
\leq \|\Lambda_l^{1-3\epsilon}\phi\|_{L^1(\RR)}\|f_0\|_{L^\infty(\RR)},
$$
$$
\int_\RR\phi(x)\left(\Lambda_l-\Lambda_l^{1-\epsilon}\right)\feps(x)dx
\leq \|\left(\Lambda_l-\Lambda_l^{1-\epsilon}\right)\phi\|_{L^1(\RR)}\|f_0\|_{L^\infty(\RR)},
$$
and
$$
\int_\RR\phi(x)\left(\Lambda_l-\Lambda_l^{1-3\epsilon}\right)\feps(x)dx
\leq \|\left(\Lambda_l-\Lambda_l^{1-3\epsilon}\right)\phi\|_{L^1(\RR)}\|f_0\|_{L^\infty(\RR)}.
$$
Using 
$$
\phi(x)-\phi(x-\eta)-\eta\pax\phi(x)=\eta^2\int_0^1\int_0^1(s-1)\pax^2\phi(x+r(s-1)\eta)drds,
$$
we bound the linear terms in \eqref{IVsisder2} as
\begin{multline*}
\|\left(\Lambda_l-\Lambda_l^{1-\epsilon}\right)\feps\|_{W^{-2,\infty}_*(B(0,N))}+\|\left(\Lambda_l-\Lambda_l^{1-3\epsilon}\right)\feps\|_{W^{-2,\infty}_*(B(0,N))}
\\
+\|\Lambda_l^{1-3\epsilon}\feps\|_{W^{-2,\infty}_*(B(0,N))}+\|\Lambda_l^{1-\epsilon}\feps\|_{W^{-2,\infty}_*(B(0,N))}\\
+\|\pax^2\feps\|_{W^{-2,\infty}_*(B(0,N))}+\|\feps\|_{W^{-2,\infty}_*(B(0,N))}\leq c\|f_0\|_{L^\infty(\RR)},
\end{multline*}
being $c$ a universal constant. The nonlinear terms are
$$
I_1=\int_\RR\phi(x)\partial_x\text{P.V.}\left(\int_{B(0,1)}+\int_{B^c(0,1)}\right)
\arctan\left(\mu_1(t)\left|\tanh\left(\frac{\eta}{2}\right)\right|^\epsilon\right)d\eta dx=J_1+J_2,
$$
and
$$
I_2=\int_\RR\phi(x)\partial_x\text{P.V.}\left(\int_{B(0,1)}+\int_{B^c(0,1)}\right)
\arctan\left(\frac{\mu_2(t)}{\left|\tanh\left(\frac{\eta}{2}\right)\right|^\epsilon}\right)d\eta dx=J_3+J_4.
$$
Using the boundedness of $\arctan$, we get
$$
|J_i|\leq \pi\|\pax \phi\|_{L^1(\RR)},\text{ for }i=1,3.
$$
The outer part is not singular and can be bounded (as it was done before) applying $\epsilon<1/10$. We get
$$
|J_i|\leq c\|\phi\|_{L^1(\RR)}\left(\tan\left(\|f_0\|_{L^\infty(\RR)}\right)+1\right),\text{ for }i=2,4.
$$
Putting together all these bounds we obtain
$$
\sup_{t\in[0,T]}\|\pat f(t)\|_{W^{-2,\infty}_*(B(0,N))}\leq C\left(\|f_0\|_{L^\infty(\RR)}\right).
$$
Using Lemma \ref{IVlemapaco}, we conclude the result.
\end{proof}

\section{Convergence of the regularized system}\label{IVsec5}
Looking at \eqref{IVeqderiv} we give the following definition
\begin{defi}\label{IVdefi}
$f(x,t)\in C([0,T]\times\RR)\cap L^\infty([0,T],W^{1,\infty}(\RR))$ is a weak solution of \eqref{IVeqderiv} if, for all $\phi(x,t)\in C_c^\infty([0,T)\times \RR)$ the following equality holds
\begin{multline*}
\int_0^T\int_\RR f(x,t)\pat\phi(x,t)dxdt+\int_\RR f_0(x)\phi(x,0)dx\\=\frac{\rho^2-\rho^1}{2\pi}\int_0^T\int_\RR\pax \phi(x,t)\left[\text{P.V.}\int_\RR\arctan\left(\frac{\tan\left(\frac{\pi}{2l}\frac{f(x)-f(x-\eta)}{2}\right)}{\tanh\left(\frac{\pi}{2l}\frac{\eta}{2}\right)}\right)d\eta\right.\\ \left.
+\text{P.V.}\int_\RR\arctan\left(\tan\left(\frac{\pi}{2l}\frac{f(x)+f(x-\eta)}{2}\right)\tanh\left(\frac{\pi}{2l}\frac{\eta}{2}\right)\right)d\eta \right]dxdt.
\end{multline*}
\end{defi}
In this section we show the convergence, as $\epsilon\rightarrow0$, of the weak formulation (see Definition \ref{IVdefi}) of the problem \eqref{IVeqreg}. 
\begin{prop}
Let $f$ be the limit of the regularized solutions $\feps$. Then $f$ is a weak solution of \eqref{IVeqderiv}.
\end{prop}
\begin{proof}
First, we deal with the linear terms. Using the weak-* convergence in $L^\infty([0,T],W^{1,\infty}(\RR))$ and Lemma \ref{IVops}, we obtain
$$
\int_0^T\int_\RR \feps(x,t)\pat\phi(x,t)dxdt\rightarrow\int_0^T\int_\RR f(x,t)\pat\phi(x,t)dxdt,
$$
$$
\int_0^T\int_\RR \feps(x,t)\phi(x,t)dxdt\rightarrow\int_0^T\int_\RR f(x,t)\phi(x,t)dxdt,
$$
$$
\int_0^T\int_\RR \feps(x,t)\Lambda_l^{1-\epsilon}\phi(x,t)dxdt\rightarrow\int_0^T\int_\RR f(x,t)\Lambda_l\phi(x,t)dxdt,
$$
$$
\int_0^T\int_\RR \feps(x,t)\Lambda_l^{1-3\epsilon}\phi(x,t)dxdt\rightarrow\int_0^T\int_\RR f(x,t)\Lambda_l\phi(x,t)dxdt,
$$
and
$$
\int_\RR \feps_0(x)\phi(x,0)dx\rightarrow \int_\RR f_0(x)\phi(x,0)dx,
$$
where, in the last step, we use the dominated convergence theorem and the $L^1$ convergence of the mollifier. To deal with the nonlinear terms we split the integrals
$$
\text{P.V.}\int_\RR=\text{P.V.}\int_{B(0,\delta)}+\text{P.V.}\int_{B^c(0,\delta)\cap B(0,N)}+\text{P.V.}\int_{B^c(0,N)},
$$
for sufficiently small $\delta$ and large enough $N$. These parameters, $\delta,N$, that will be fixed below, can depend on $f_0$ but they don't depend on $\epsilon$. For the inner part of the integrals, we get
\begin{multline*}
I^\epsilon_1=\int_0^T\int_\RR\pax \phi(x,t)\left(2\text{P.V.}\int_{B(0,\delta)}\arctan\left(\mu_1(t)\left|\tanh(\eta/2)\right|^\epsilon\right)d\eta\right.\\ 
\left.+2\text{P.V.}\int_{B(0,\delta)}\arctan\left(\frac{\mu_2(t)}{\left|\tanh(\eta/2)\right|^\epsilon}\right)d\eta \right)dxdt\leq c\delta\|\pax \phi\|_{L^1([0,T]\times\RR)}.
\end{multline*}

The outer integral goes to zero as $N$ grows. We compute
\begin{multline*}
I^\epsilon_3=\int_0^T\int_{\RR}\pax \phi(x,t)\left(2\text{P.V.}\int_{B^c(0,N)}\arctan\left(\mu_1(t)\left|\tanh(\eta/2)\right|^\epsilon\right)d\eta\right.\\ 
\left.+2\text{P.V.}\int_{B^c(0,N)}\arctan\left(\frac{\mu_2(t)}{\left|\tanh(\eta/2)\right|^\epsilon}\right)d\eta \right)dxdt
\end{multline*}
As $\eta\in B^c(0,N)$, the integrals are not singular and we only have to deal with the decay at infinity. Using \eqref{IVfact1}, \eqref{IVfact1.b}, \eqref{IVeqreg2}, the bound $\epsilon<1/10$, integrating by parts and using the extra decay coming from the principal value at infinity (see, for instance, the term $A_6$ in Proposition \ref{existence} in Section \ref{IVsecglobal}), we have
$$
I^\epsilon_3\rightarrow0,\text{ uniformly in $\epsilon$ as }N\rightarrow\infty.
$$

The only thing to check is the convergence of $I_2^\epsilon$. Due to the compactness of the support of $\phi$, we have
\begin{multline*}
I^\epsilon_2=\int_0^T\int_\RR\pax \phi(x,t)\left(2\text{P.V.}\int_{B^c(0,\delta)\cap B(0,N)}\arctan\left(\mu_1(t)\left|\tanh(\eta/2)\right|^\epsilon\right)d\eta\right.\\ 
\left.+2\text{P.V.}\int_{B^c(0,\delta)\cap B(0,N)}\arctan\left(\frac{\mu_2(t)}{\left|\tanh(\eta/2)\right|^\epsilon}\right)d\eta \right)dxdt\\
=\int_0^T\int_{B(0,M)}\pax \phi(x,t)\left(2\text{P.V.}\int_{B^c(0,\delta)\cap B(0,N)}\arctan\left(\mu_1(t)\left|\tanh(\eta/2)\right|^\epsilon\right)d\eta\right.\\ 
\left.+2\text{P.V.}\int_{B^c(0,\delta)\cap B(0,N)}\arctan\left(\frac{\mu_2(t)}{\left|\tanh(\eta/2)\right|^\epsilon}\right)d\eta \right)dxdt,
\end{multline*}
with $M$ large enough to ensure $\text{supp}(\phi)\subset B(0,M)$. Since we have (up to a subsequence) that $\feps\rightarrow f$ uniformly on compact sets (see Lemma \ref{IVlemafeps}), the uniform convergence $|\tanh(\eta/2)|^ \epsilon\rightarrow 1$ if $|\eta|>\delta$ and the continuity of all the functions in this integral, the limit in $\epsilon$ and the integral commute and we get
\begin{multline*}
I^\epsilon_2\rightarrow\int_0^T\int_\RR\pax \phi(x,t)\left(2\text{P.V.}\int_{B^c(0,\delta)\cap B(0,N)}\arctan\left(\frac{\tan\left(\frac{f(x)-f(x-\eta)}{2}\right)}{\tanh\left(\frac{\eta}{2}\right)}\right)d\eta\right.\\ 
\left.+2\text{P.V.}\int_{B^c(0,\delta)\cap B(0,N)}\arctan\left(\tan\left(\frac{f(x)+f(x-\eta)}{2}\right)\tanh\left(\frac{\eta}{2}\right)\right)d\eta \right)dxdt=I_2^0.
\end{multline*}

We conclude the proof of the Theorem \ref{IVglobal} by taking $\delta<<1$ and $N>>1$ to control the tails and then we send $\epsilon\rightarrow0$. 
\end{proof}

\bibliographystyle{abbrv}

\end{document}